\documentclass[reqno,11pt]{amsart}
\usepackage{amsfonts}
\usepackage{a4wide}
\usepackage{color}
\usepackage{mathrsfs}
\usepackage{mathtools}
\usepackage{amsmath}
\usepackage{amssymb}
\usepackage{bbm}
\usepackage{esint}
\usepackage{nicefrac}
\numberwithin{equation}{section}
\usepackage[colorlinks,citecolor=green,linkcolor=red]{hyperref}

\usepackage[latin1]{inputenc}
\input xy
\xyoption{all}

\newtheorem{theorem}{Theorem}[section]

\newtheorem{Th}[theorem]{Theorem}
\newtheorem{Lm}[theorem]{Lemma}
\newtheorem{Prop}[theorem]{Proposition}
\newtheorem{Def}[theorem]{Definition}

\newtheorem{Remark}[theorem]{Remark}

\title{Characterization of AC and Sobolev curves via Lipschitz post-compositions}
\author{Roman D. Oleinik and Alexander I. Tyulenev}
\address{Steklov Mathematical Institute of Russian academy of Sciences}
\email{oleinik.r@phystech.edu;
tyulenev-math@yandex.ru,tyulenev@mi-ras.ru}
\begin{document}
\date{\today}
\allowdisplaybreaks
\keywords{Metric spaces, absolutely continuous curves}
\subjclass[2010]{53C23, 46E35}
\begin{abstract}
Let $\operatorname{X}:=(\operatorname{X},\operatorname{d})$ be an arbitrary metric space.
For each $p \in [1,\infty]$, we prove that a map $\gamma:[a,b] \to \operatorname{X}$ is $p$-absolutely continuous if and only if, for every
Lipschitz function $h:\operatorname{X} \to \mathbb{R}$, the post-composition $h \circ \gamma$
is a $p$-absolutely continuous function. Furthermore,
if $\operatorname{X}$ is complete and separable, then, for each $p \in (1,\infty)$, we show that the equivalence class (up to $\mathcal{L}^{1}$-a.e.\ equality) of a Borel map $\gamma:[a,b] \to \operatorname{X}$
belongs to the Sobolev $W_{p}^{1}([a,b],\operatorname{X})$-space if and only if,
for every
Lipschitz function $h:\operatorname{X} \to \mathbb{R}$, the equivalence class (up to $\mathcal{L}^{1}$-a.e.\ equality) of the post-composition $h \circ \gamma $ belongs to the Sobolev $W_{p}^{1}([a,b],\mathbb{R})$-space.
\end{abstract}
\maketitle
\tableofcontents
\section{Introduction}
\noindent

The study of many regularity properties of a given map $\gamma:[a,b] \to \mathbb{R}^{n}$ can be easily reduced to the same regularity properties of real-valued functions by taking post-compositions of $\gamma$ with coordinate functions. In other words,
there is a special finite family of $1$-Lipschitz functions intimately related to the structure of the space that can be used to control the behavior of $\gamma$.
However, if we replace $\mathbb{R}^{n}$ by an arbitrary metric space $(\operatorname{X},\operatorname{d})$, the situation
changes drastically. Indeed, due to the lack of any additional structure, the only substitution for the role of ``coordinates on $\operatorname{X}$'' is
the family of all $1$-Lipschitz real-valued functions on $\operatorname{X}$. This family is typically infinite, as well as there is no
canonical finite collection within it that can be used as a keystone for subsequent analysis.

The main goal of this paper is to show that, given an abstract metric space $\operatorname{X}=(\operatorname{X},\operatorname{d})$, some crucial regularity properties of a generic map $\gamma:[a,b] \to \operatorname{X}$ can be extracted from the same regularity properties of its post-compositions with Lipschitz functions.

The first regularity property we are aimed to attack is $p$-absolute continuity. Recall that, given $p \in [1,\infty]$ and an arbitrary metric space
$\operatorname{X}=(\operatorname{X},\operatorname{d})$, a map $\gamma:[a,b] \to \operatorname{X}$ is said to be \textit{$p$-absolutely continuous} if there
is a nonnegative function $g \in L_{p}([a,b])$ such that
$$
\operatorname{d}(\gamma(t_{1}),\gamma(t_{2})) \le \int\limits_{t_{1}}^{t_{2}}g(\tau)\,d\tau  \quad \hbox{for all} \quad a \le t_{1} \le t_{2} \le b.
$$
The family of all $p$-absolutely continuous maps $\gamma:[a,b] \to \operatorname{X}$ will be denoted by $AC_{p}([a,b],\operatorname{X})$. We also set
$AC_{p}([a,b]):=AC_{p}([a,b],\mathbb{R})$ for brevity.

We let $\operatorname{LIP}(\operatorname{X})$ denote the space of all Lipschitz functions $h:\operatorname{X} \to \mathbb{R}$.
It is well known and not difficult to show (see Lemma \ref{Lm.elementary_characterization} below) that, given $p \in [1,\infty]$ and a metric space $\operatorname{X}=(\operatorname{X},\operatorname{d})$,
a map $\gamma: [a,b] \to \operatorname{X}$ belongs to the space $AC_{p}([a,b],\operatorname{X})$ if and only if
there is a nonnegative $g \in L_{p}([a,b])$ such that for every $h \in \operatorname{LIP}(\operatorname{X})$ the composition $h \circ \gamma$ belongs to
$AC_{p}([a,b])$ and (here and below, $\mathcal{L}^{n}$ means the classical $n$-dimensional Lebesgue measure)
\begin{equation}
\label{eqq.AC_ineq}
\Bigl|\frac{d(h \circ \gamma)}{dt}(t)\Bigr| \le \operatorname{lip}h(\gamma(t)) g(t) \quad \hbox{for $\mathcal{L}^{1}$-a.e. $t \in [a,b]$},
\end{equation}
where $\operatorname{lip}h(\gamma(t))$ is the so-called local
Lipschitz constant (or, equivalently, the slope of $h$) at the point $\gamma(t)$.
Furthermore, there exists a minimal in the $\mathcal{L}^{1}$-a.e.\ sense function $g$ for which \eqref{eqq.AC_ineq} holds, called the metric speed of $\gamma$
and denoted by $|\dot{\gamma}|$.

One can put this observation  as a keystone for the possible definition of Sobolev and BV maps with values in arbitrary metric spaces.
This was firstly done by L.~Ambrosio in \cite{Amb} for the BV-case and then by Yu.~G.~Reshetnyak in \cite{Re} for the Sobolev case.
Before we briefly recall the corresponding definition, we should warn the reader that throughout
the paper, given a metric space $\operatorname{Y}=(\operatorname{Y},\rho)$, we will always distinguish a map $f:[a,b] \to \operatorname{Y}$
from its equivalence (up to $\mathcal{L}^{1}$-a.e.\ equality on $[a,b]$) class $[f]$. All inequalities involving equivalence classes of maps should be understand to hold
for arbitrary representatives of the corresponding classes.

Combining Lemma 2.13 in \cite{GT1} with Corollary 3.10 in \cite{GT2} we can reformulate Reshetnyak's definition of Sobolev curves as follows. Given a complete separable metric
space $\operatorname{X}=(\operatorname{X},\operatorname{d})$ and a parameter $p \in (1,\infty)$, we say that the equivalence class $[\gamma]$ (modulo $\mathcal{L}^{1}$-a.e.\ coincidence on $[a,b]$) of a
Borel map $\gamma:[a,b] \to \operatorname{X}$ belongs to the Sobolev space $W_{p}^{1}([a,b],\operatorname{X})$ if $[\gamma] \in L_{p}([a,b],\operatorname{X})$, $[h \circ \gamma] \in W_{p}^{1}([a,b])$ for every  $h \in \operatorname{LIP}(\operatorname{X})$, and there is
a nonnegative Borel function $G$ with $[G] \in L_{p}([a,b])$ called a \textit{$p$-weak upper gradient of $[\gamma]$} such that
\begin{equation}
\label{eqq.Sobolev_ineq}
\bigl|\partial [h \circ \gamma](t)\bigr| \le [\operatorname{lip}(h \circ \gamma)](t) [G](t) \quad \hbox{for $\mathcal{L}^{1}$-a.e. $t \in [a,b]$},
\end{equation}
where $\partial [h \circ \gamma]$ is the usual Sobolev distributional derivative of $[h\circ\gamma]$. Furthermore, there is a minimal (in the $\mathcal{L}^{1}$-a.e.\ sense) function $G$ for which \eqref{eqq.Sobolev_ineq} holds, the equivalence class $[G]$ is called the \textit{distributional
derivative} of $[\gamma]$,
and is denoted by $|\partial{[\gamma]}|$ (see Theorem 2.11 in \cite{GT1}).

In other words,  given $p \in [1,\infty]$, in order to determine whether a Borel map $\gamma:[a,b] \to \operatorname{X}$ or its class $[\gamma]$
belongs to $AC_{p}([a,b],\operatorname{X})$-space or $W_{p}^{1}([a,b],\operatorname{X})$-space, respectively, one should establish the existence of a some
sort of an $L_{p}$-majorant for the derivatives of post-compositions of $\gamma$ with arbitrary Lipschitz functions.
We show that one can \textit{remove quantitative conditions} from the definitions, i.e.
the existence of the corresponding majorants should automatically follow from the ``qualitative'' fact that for every
$h \in \operatorname{LIP}(\operatorname{X})$ we have $h \circ \gamma \in AC_{p}([a,b])$ or $[h \circ \gamma] \in W_{p}^{1}([a,b])$, respectively.
\textit{The first main result} of the present paper is a characterization of $p$-absolutely continuous curves via Lipschitz post-compositions.

\begin{Th}
\label{Th.main_1}
Let $\operatorname{X}=(\operatorname{X},\operatorname{d})$ be a metric space. Given $p \in [1,\infty]$, a map $\gamma:[a,b] \to \operatorname{X}$ belongs
to the space  $AC_{p}([a,b],\operatorname{X})$ if and only if $h \circ \gamma \in AC_{p}([a,b])$ for every $h \in \operatorname{LIP}(\operatorname{X})$.
\end{Th}

The proof of the ``necessity part'' is clear.
The strategy of the proof of the ``sufficiency part'' follows the scheme:

\begin{itemize}

\item[\((\textbf{ST}1)\)] We show that if $h \circ \gamma \in C([a,b])$ for every $h \in \operatorname{LIP}(\operatorname{X})$, then
$\gamma \in C([a,b],\operatorname{X})$;

\item[\((\textbf{ST}2)\)]
We show that if the post-composition $h \circ \gamma$ has  bounded variation for every $h \in \operatorname{LIP}(\operatorname{X})$,
then the image $\Gamma:=\gamma([a,b])$ has finite $\mathcal{H}^{1}$-measure in $\operatorname{X}$;

\item[\((\textbf{ST}3)\)]  Assuming that $h \circ \gamma$ has  bounded variation for every $h \in \operatorname{LIP}(\operatorname{X})$, we use the classical
$1$-rectifiability result for curves (see, for example, Theorem 4.4.5 in \cite{AmbTil}) to show that $\Gamma:=\gamma([a,b])$ is $1$-rectifiable;

\item[\((\textbf{ST}4)\)] For each $\epsilon > 0$, we built a Lipschitz function $h_{\epsilon}$ whose Jacobian with respect to $\Gamma$ equals $1$ on a set
$\Gamma_{\epsilon} \subset \Gamma$ with $\mathcal{H}^{1}(\Gamma \setminus \Gamma_{\epsilon}) < \epsilon$;

\item[\((\textbf{ST}5)\)] Assuming that $h \circ \gamma$ has  bounded variation for every $h \in \operatorname{LIP}(\operatorname{X})$,
we apply the area formula in combination with integral representation for a variation to establish that $\gamma$ has
bounded variation;

\item[\((\textbf{ST}6)\)] If $h \circ \gamma \in AC_{1}([a,b])$ for every $h \in \operatorname{LIP}(\operatorname{X})$, then
combining the Banach--Zarecki-type characterization of $AC_{1}$-curves with (\textbf{ST}3)--(\textbf{ST}5) we deduce that $\gamma \in AC_{1}([a,b],\operatorname{X})$.

\item[\((\textbf{ST}7)\)] If $h \circ \gamma \in AC_{p}([a,b])$ for every $h \in \operatorname{LIP}(\operatorname{X})$, then
we combine (\textbf{ST}4) and (\textbf{ST}6) with some standard arguments from the infinite dimensional analysis to deduce that $\gamma$ belongs to $AC_{p}([a,b],\operatorname{X})$.

\end{itemize}

We believe that step (\textbf{ST}4) is the most complicated part of the proof. We would like to give some explanations.
Furthermore, in this paper we give two alternative ways for that step, ``inderect way''  and ``direct way'', respectively.

\textit{The first} ``inderect way'' for (\textbf{ST}4) relies on the beautiful powerful results from \cite{Bate}. Namely, using Theorem 7.6 of that paper we can make
the following observation, which can be interesting in itself.
We denote by $\operatorname{BLIP}_{1}(\operatorname{X})$ the set of
all bounded 1-Lipschitz functions $h:\operatorname{X} \to \mathbb{R}$,
equipped with the supremum norm, a complete metric space.
Recall that a set $S \subset \operatorname{BLIP}_{1}(\operatorname{X})$ is  \textit{residual} if it contains a countable
intersection of open dense subsets. Below, the symbol $\mathcal{J}_{\Gamma}h$ means the Jacobian of a function $h \in \operatorname{LIP}(\operatorname{X})$
with respect to $\Gamma$ (see Definition \ref{Def.Jacobian} for details).
Furthermore, given a metric space $\operatorname{Y}=(\operatorname{Y},\rho)$ and a map $g:[a,b] \to \operatorname{Y}$, the symbol $V_{g}([a,b])$
means the variation of $g$ on $[a,b]$. Finally, by $BV([a,b],\operatorname{Y})$ we denote the set of all maps $g:[a,b] \to \operatorname{Y}$ with $V_{g}([a,b]) < +\infty$ (see Definition \ref{Def.1} below).

\begin{Th}
\label{Th.main_Bate}
Let $\operatorname{X}=(\operatorname{X},\operatorname{d})$ be a nonempty metric space. If a map $\gamma: [a,b] \to \operatorname{X}$ is such that $h \circ \gamma \in C([a,b]) \cap BV([a,b])$ for every $h \in \operatorname{LIP}(\operatorname{X})$, then $\gamma \in C([a,b],\operatorname{X}) \cap BV([a,b],\operatorname{X})$ and there exists a residual set $S \subset \operatorname{BLIP}_{1}(\operatorname{X})$ such that,
for each $h \in S$, $\mathcal{J}_{\Gamma}h(x)=1$ for $\mathcal{H}^{1}$-a.e. $x \in \Gamma$ and, for each $h \in S$,
\begin{equation}
\label{eqq.variation}
V_{h \circ \gamma}([c,d]) = V_{\gamma}([c,d]) \quad \hbox{for all} \quad a \le c \le d \le b.
\end{equation}
\end{Th}
While the use of a machinery from \cite{Bate} in combination with our (\textbf{ST}1)--(\textbf{ST}3)
makes the proof of Theorem \ref{Th.main_Bate} quite short, the result can be considered as an ``existence theorem''. Indeed, the proof of the
crucial Theorem 7.6 in \cite{Bate} is quite complicated and not fully constructive.

\textit{The second} ``direct way'' is longer but more transparent.
We believe that it can be interesting in itself.
In contrast to \cite{Bate}, using quite elementary methods,
we present an explicit elementary
procedure for construction of a ``good sawtooth'' Lipschitz function concentrated on the image $\Gamma=\gamma([a,b])$ of a given rectifiable curve $\gamma \in C([a,b],\operatorname{X})$.
The Jacobian (with respect to $\Gamma$) of that function
equals $1$ on a set of an ``almost full measure''.
Despite the fact that this result is weaker than Theorem \ref{Th.main_Bate}, it is sufficient for
the proof of Theorem \ref{Th.main_1}.

The base strategy for steps (\textbf{ST}2) and (\textbf{ST}5)--(\textbf{ST}7) is as follows.
We assume that some regularity property of a given curve fails and then we construct a concrete function $h \in \operatorname{LIP}(\operatorname{X})$ such that
the corresponding regularity property fails for the post-composition $h \circ \gamma$. We should underline that the use of direct and explicit methods is a specific feature
of the present paper. In some places we could use Banach--Steinhaus-type arguments to make the corresponding proofs shorter.
However, we instead construct explicit examples ``by hands''. We believe that our approach is more suitable
for possible applications.

\textit{The second main} result concerns Sobolev maps and has a similar flavor. In other words, we present a characterization
of Sobolev curves via Lipschitz post-compositions.
However, in contrast to the previous case, some reasonable restrictions on $\operatorname{X}$ should be imposed.
Indeed, if $\gamma \in C([a,b],\operatorname{X})$, then one need not
work with the whole space $\operatorname{X}$ but rather with the set $\Gamma:=\gamma([a,b]) \subset \operatorname{X}$, which becomes a complete separable metric space
after the restriction of the metric $\operatorname{d}$ to $\Gamma$. In contrast to this case, given a Sobolev element $[\gamma] \in W_{p}^{1}([a,b], \operatorname{X})$, a generic representative $\widetilde{\gamma}$ of $[\gamma]$
can have a ``bad set'' $E \subset [a,b]$
with $\mathcal{L}^{1}(E)=0$, on which its values cannot be controlled in any reasonable way. In particular, the image $\widetilde{\gamma}(E)$ can have infinite $\mathcal{H}^{1}$-measure in the space $\operatorname{X}$.
Hence, to avoid some delicate technical details, all results in this paper involving Sobolev spaces will be proved under completeness and separability assumptions of $\operatorname{X}$.

The crucial ingredient for the characterization of the Sobolev curves via Lipschitz postcompositions is the following result which may be interesting in itself.
\begin{Th}
\label{Th.representative}
Let $\gamma:[a,b] \to \operatorname{X}$ be a Borel map such that $[h\circ\gamma] \cap C([a,b]) \neq \emptyset$ for every function $h \in \operatorname{LIP}(\operatorname{X})$. Then $[\gamma] \cap C([a,b],\operatorname{X}) \neq \emptyset$.
\end{Th}

\textit{The second main result} of the present paper reads as follows.
\begin{Th}
\label{Th.main_2}
Let $\operatorname{X}=(\operatorname{X},\operatorname{d})$ be a complete separable metric space. Given $p \in (1,\infty)$, a
map $\gamma:[a,b] \to \operatorname{X}$ is Borel with the equivalence class $[\gamma]$ in $W_{p}^{1}([a,b],\operatorname{X})$ if and only if $h \circ \gamma$ is Borel and $[h\circ\gamma] \in W_{p}^{1}([a,b],\mathbb{R})$ for every $h \in \operatorname{LIP}(\operatorname{X})$.
\end{Th}

As far as we know, the problems studied in the present paper
have never been considered in the literature (at least at such a high generality).
However, very recently Professor V.~I.~Bakhtin proved Theorem \ref{Th.main_1} in the case $p=1$ \cite{Bah}.
He had known about our preliminary results from the talk of the second named author made at Belarusian State University at the end of April 2024.
At that time, the authors proved Theorem \ref{Th.main_1} in the particular case $p=1$
under the metrically doubling assumption on $\operatorname{X}$. A little bit later, the authors and Professor V.~I.~Bakhtin
independently and almost simultaneously realized that the doubling assumption on $\operatorname{X}$ can be removed.
Note that Bakhtin's proof is very elegant, short and fully direct. It differs from both (the current and the initial) proofs given by the authors. 
Moreover, his proof does not rely on any technique from the theory of rectifiable sets in metric spaces.
On the other hand, we do not know, whether that approach can be used to cover the case $p > 1$. Furthermore, the characterization of Sobolev-type
regularity of curves via Lipschitz post-compositions was not considered in \cite{Bah}.

{\bf Acknowledgements.}
First of all, we express our deep gratitude to Professor S.~K.\ Vodop'yanov whose beautiful results  and talks on composition operators in
Sobolev spaces inspired us to attack the problems considered in the present paper. We also grateful to Professors L.~Ambrosio and N.~Gigli for important conversations
concerning metric-valued curves and rectifiability of sets in metric spaces. In particular, Professor L.~Ambrosio kindly informed us that the result about $1$-rectifiability used in (\textbf{ST}3) is
valid in arbitrary metric spaces and shared the precise reference. Finally, we are grateful to Professors R.~N.~Karasev and D.~M.~Stolyarov for fruitful discussions of (\textbf{ST}2),
which helped us to simplify some considerations.

\section{Geometric measure theory background and notation.}

In this section we gather some well-known concepts and facts from the contemporary Geometric Analysis, which
will be building blocks for subsequent exposition.

\subsection{Metric spaces and measures}
Given a nonempty metric space $(\operatorname{X},\operatorname{d})$, by \textit{a ball} $B$ in $\operatorname{X}$ we always mean \textit{an open ball}
with a priori given center and radius. For each ball $B \subset \operatorname{X}$, we let $\overline{B}$ denote the closed ball with the same center and radius.
More precisely,
$$
B=B_{r}(x):=\{y \in \operatorname{X}: \operatorname{d}(x,y) < r\}, \quad \overline{B}:=\overline{B}_{r}(x):=\{y \in \operatorname{X}: \operatorname{d}(x,y) \le r\}
$$
for some $x \in \operatorname{X}$ and $r > 0$. For each set $E \subset \operatorname{X}$, the symbols $\operatorname{cl}E$ and $\operatorname{int}E$ denote the closure and the interior
of $E$, respectively.

The \textit{diameter} of a set $E \subset \operatorname{X}$ is defined by $\operatorname{diam}E:=\sup\{\operatorname{d}(x,y):x,y \in E\}$, as usual.
Given a nonempty set $E \subset \operatorname{X}$ and a parameter $\epsilon > 0$,
we say that $\mathcal{N}$ is an \textit{$\epsilon$-separated subset} of $E$ if $\operatorname{d}(x,y) \geq \epsilon$ for all $x,y \in \mathcal{N}$ with $x \neq y$.
An $\epsilon$-separated subset $\mathcal{N}$ of a set $E \subset \operatorname{X}$ is said to be \textit{maximal} if $\operatorname{dist}(x,\mathcal{N}) < \epsilon$ for every $x \in E$.

Given a metric space $(\operatorname{X},\operatorname{d})$ and a nonempty set $E \subset \operatorname{X}$, we will sometimes consider $E$ as a metric space
whose metric $\operatorname{d}_{E}$ is the restriction of the ambient metric $\operatorname{d}$ to $E$.

If $E$ is an abstract set, then by $\# E$ we denote its cardinality, i.e., the number of different elements in $E$. In other words,
$\# E =0$ if and only if $E = \emptyset$, and $\# E = +\infty$ if and only if $E$ is an infinite set.

Given metric spaces $\operatorname{X}=(\operatorname{X},\operatorname{d})$ and $\operatorname{Y}=(\operatorname{Y},\rho)$, by $\mathfrak{B}(\operatorname{X},\operatorname{Y})$ we denote the
set of all Borel maps $F$ from $\operatorname{X}$ to $\operatorname{Y}$. We also put $\mathfrak{B}(\operatorname{X}):=\mathfrak{B}(\operatorname{X},\mathbb{R})$ for brevity.
Given a map $F:\operatorname{X} \to \operatorname{Y}$ and a nonempty set $E \subset \operatorname{X}$, by $F|_{E}$ we denote the pointwise restriction of $F$ to $E$.

The symbol $\mathcal{L}^{n}$, $n \in \mathbb{N}$, denotes the classical $n$-dimensional Lebesgue measure in $\mathbb{R}^{n}$, as usual.
Let $\operatorname{X}=(\operatorname{X},\operatorname{d})$ be a metric space, $E \subset \operatorname{X}$ be an arbitrary subset of $\operatorname{X}$, and $d \geq 0$.
For each $\delta >0$, the \textit{$d$-Hausdorff content} of $E$ at the scale $\delta$ is defined by
\begin{equation}
\label{eqq.1}
\mathcal{H}^{d}_{\delta}(E):=\inf\sum\limits_{k=1}^{\infty}\Bigl(\operatorname{diam}U_{k}\Bigr)^{d},
\end{equation}
where the infimum is taken over all countable coverings $\mathcal{U}=\{U_{k}\}_{k=1}^{\infty}$ of $E$ by arbitrary sets $U_{k} \subset \operatorname{X}$ such that $\operatorname{diam}U_{k} < \delta$ for all $k \in \mathbb{N}$.
The \textit{$d$-Hausdorff measure} of $E$ is defined by
\begin{equation}
\label{eqq.2}
\mathcal{H}^{d}(E):=\sup\limits_{\delta > 0}\mathcal{H}^{d}_{\delta}(E)=\lim\limits_{\delta \to +0}\mathcal{H}^{d}_{\delta}(E).
\end{equation}

\textbf{Important notation.}
Given a metric space $\operatorname{Y}=(\operatorname{Y},\rho)$ and a map $f \in \mathfrak{B}([a,b],\operatorname{Y})$, by $[f]$ we denote the equivalence
class (up to $\mathcal{L}^{1}$-a.e.\ equality on $[a,b]$) of $f$. More precisely,
\begin{equation}
\label{eqq.equivalence_class}
[f]:=\{\widetilde{f}: [a,b] \to \operatorname{Y}:\mathcal{L}^{1}(\{t \in [a,b]: \widetilde{f}(t) \neq f(t)\})=0\}.
\end{equation}
Throughout the paper we will always \textit{distinguish between} a map and its equivalence class. Furthermore, given
$f_{1},f_{2} \in \mathfrak{B}([a,b],\operatorname{Y})$, the formula
$$
[f_{1}](t) = [f_{2}](t) \quad \hbox{for $\mathcal{L}^{1}$-a.e. $t \in [a,b]$}
$$
\textit{should be interpreted} in such a way that, for every representatives $\widetilde{f}_{1} \in [f_{1}]$ and
$\widetilde{f}_{2} \in [f_{2}]$, the equality $\widetilde{f}_{1}(t)=\widetilde{f}_{2}(t)$ holds true
for $\mathcal{L}^{1}$-a.e. $t \in [a,b]$. In the case $\operatorname{Y}=(\mathbb{R},|\cdot|)$, a \textit{similar interpretation}
of the ``inequalities'' for the equivalence classes of functions should be used.

Keeping in mind the above notation, we will be very careful with the usual Lebesgue spaces. More precisely, for each
$p \in [1,\infty)$, by $\widetilde{L}_{p}([a,b])$ we mean the linear space of all functions $f \in \mathfrak{B}([a,b])$
such that $\int_{a}^{b}|f(t)|^{p}\,dt < +\infty$. By $\widetilde{L}_{\infty}([a,b])$ we denote the linear space of all essentially bounded
Borel functions. Given a metric space $(\operatorname{Y},\rho)$ and a parameter $p \in [1,\infty]$, by $\widetilde{L}_{p}([a,b],\operatorname{X})$ we denote the
set of all maps $f \in \mathfrak{B}([a,b],\operatorname{Y})$ such that $\rho(\underline{y},f(\cdot)) \in \widetilde{L}_{p}([a,b])$ for some (and hence every) $\underline{y} \in \operatorname{Y}$.
Finally, for each $p \in [1,\infty]$, we put
$$
L_{p}([a,b],\operatorname{Y}):=\{[f]:f \in \widetilde{L}_{p}([a,b],\operatorname{Y})\}
$$
and equip the corresponding spaces with the usual metric (norm in the case $\operatorname{Y}=(\mathbb{R},|\cdot|)$).

\subsection{Curves and variations}

Given a closed interval $[a,b]$, by \textit{a partition of $[a,b]$} we always mean a finite linearly ordered subset of $[a,b]$ starting at $a$ and ending at $b$. More precisely,
we put $\operatorname{T}([a,b]):=\{t_{i}\}_{i=0}^{N}$, $N \in \mathbb{N}$, and $a=t_{0} < ... <t_{N}=b$. By $\mathcal{T}([a,b])$ we denote the family of all partitions of $[a,b]$.
If $\operatorname{T}=\{t_{i}\}_{i=0}^{N} \in \mathcal{T}([a,b])$, then
$$
\delta(\operatorname{T}):=\max\{|t_{i}-t_{i-1}|:i=1,...,N\}.
$$

Given two partitions $\operatorname{T}^{1}=\{t^{1}_{i}\}_{i=0}^{N^{1}} \in \mathcal{T}([a,b])$
and $\operatorname{T}^{2}:=\{t^{2}_{i}\}_{i=0}^{N^{2}} \in \mathcal{T}([a,b])$, \textit{their union} $\operatorname{T}^{1} \cup \operatorname{T}^{2}$ is defined
as a partition $\operatorname{T} \in \mathcal{T}([a,b])$ obtained as the union $\{t^{1}_{i}\}_{i=0}^{N^{1}} \cup \{t^{2}_{i}\}_{i=0}^{N^{2}}$ equipped with the corresponding
linear order.
We will use the following \textit{partial order} on $\mathcal{T}([a,b])$. Given  $\operatorname{T}^{1}=\{t^{1}_{i}\}_{i=0}^{N^{1}} \in \mathcal{T}([a,b])$
and $\operatorname{T}^{2}:=\{t^{2}_{i}\}_{i=0}^{N^{2}} \in \mathcal{T}([a,b])$, we write $\operatorname{T}^{1} \preceq \operatorname{T}^{2}$ if $\{t^{1}_{i}\}_{i=0}^{N^{1}} \subset \{t^{2}_{i}\}_{i=0}^{N^{2}}$.

\begin{Def}
\label{Def.2}
By \textit{a curve in} $\operatorname{X}$ we always mean a map $\gamma \in C([a,b],\operatorname{X})$ for some $a < b$. We say that a curve $\gamma$ is \textit{simple} if the map $\gamma:[a,b] \to \operatorname{X}$
is injective.
\end{Def}

\textbf{Important notation.} Given a curve $\gamma \in C([a,b],\operatorname{X})$, we put $\Gamma:=\gamma([a,b])$.
If $\gamma$ is a simple rectifiable curve and $x,y \in \Gamma$, then $\gamma_{x,y}$ denotes the subcurve of $\gamma$ joining  $x$ and $y$.

\begin{Def}
\label{Def.1}
Given a metric space $\operatorname{Y}=(\operatorname{Y},\rho)$ and a map $g:[a,b] \to \operatorname{Y}$, for each partition $\operatorname{T}=\{t_{i}\}_{i=0}^{N} \in \mathcal{T}([a,b])$,
the \textit{variation
of $g$ over $\operatorname{T}$} is defined by
\begin{equation}
\label{eq.2}
V_{g}(\operatorname{T}):=\sum\limits_{i=1}^{N}\rho(g(t_{i-1}),g(t_{i})).
\end{equation}
The \textit{variation of $g$ on $[a,b]$} is defined by
\begin{equation}
\label{eq.2'}
V_{g}([a,b]):=\sup\limits_{\operatorname{T} \in \mathcal{T}([a,b])}V_{g}(\operatorname{T}).
\end{equation}
A map $g$ is \textit{of bounded variation} (written $g \in BV([a,b],\operatorname{Y})$) if $V_{g}([a,b]) < +\infty$.
\end{Def}

In the case $\operatorname{Y}=(\mathbb{R},|\cdot|)$, we use notation $BV([a,b])$ omitting the target space.

\begin{Remark}
\label{Rem.monotonicity_variation}
By the triangle inequality,  $V_{g}(\operatorname{T}^{1}) \preceq V_{g}(\operatorname{T}^{2})$ for every $\operatorname{T}^{1},\operatorname{T}^{2} \in \mathcal{T}([a,b])$ with $\operatorname{T}^{1} \preceq \operatorname{T}^{2}$.
\end{Remark}

\begin{Def}
\label{Def.3}
We say that a curve $\gamma \in C([a,b],\operatorname{X})$ is \textit{rectifiable}, if $\gamma \in BV([a,b],\operatorname{X})$.
In this case, by \textit{the length of} $\gamma$ we mean its variation on $[a,b]$ and denote it by $l(\gamma)$.
\end{Def}

\begin{Remark}
\label{Rem.1}
It is well known (see, for example,  \S 2.7 in \cite{BBI}) that if $\gamma \in C([a,b],\operatorname{X})$ is a simple curve, then $\mathcal{H}^{1}(\Gamma)=l(\gamma)$ (the equality is considered in the range
$[0,+\infty]$). If $\gamma$ is rectifiable and
$\gamma_{s}$ its arc length parametrization, then this fact can be used  to show that, given a set $E \subset \Gamma$,  $\mathcal{H}^{1}(E)=0$ if and only if $\mathcal{L}^{1}(\gamma^{-1}_{s}(E))=0$.
\end{Remark}

The following fact is probably a folklore. Since we are not able to give the precise reference, we present
the details (we recall the important notation).

\begin{Lm}
\label{Lm.inner_and_outer_metric}
Let $\gamma \in C([a,b],\operatorname{X})$ be a simple rectifiable curve. Then there exists a set $\underline{\Gamma} \subset \Gamma$ such that $\mathcal{H}^{1}(\Gamma \setminus \underline{\Gamma}) = 0$ and
\begin{equation}
\label{eqq.4.18}
\lim\limits_{\substack{x_{1},x_{2} \to x \\ x_{1},x_{2} \in \Gamma}}\frac{l(\gamma_{x_{1},x_{2}})}{\operatorname{d}(x_{1},x_{2})} = 1  \quad \hbox{for every} \quad x \in \underline{\Gamma}.
\end{equation}
\end{Lm}

\begin{proof}
It is well known (see, for example, Theorem 2.7.4 in \cite{BBI}) that, given a rectifiable curve $\widetilde{\gamma}:[c,d] \to \operatorname{X}$, for $\mathcal{L}^{1}$-a.e. $t_{0} \in [c,d]$ we have
\begin{equation}
\label{eqq.4.19}
\hbox{either} \quad \varliminf\limits_{\varepsilon,\varepsilon' \to +0}\frac{l(\widetilde{\gamma}|_{[t_{0}-\varepsilon,t_{0}+\varepsilon']})}{\varepsilon + \varepsilon'}=0 \quad \hbox{or} \quad
\lim\limits_{\varepsilon,\varepsilon' \to +0}\frac{l(\widetilde{\gamma}|_{[t_{0}-\varepsilon,t_{0}+\varepsilon']})}{\operatorname{d}(\widetilde{\gamma}(t_{0}-\varepsilon),\widetilde{\gamma}(t_{0}+\varepsilon'))}=1.
\end{equation}

Since $\gamma$ is injective and rectifiable, there is an arc-length parametrization of $\gamma$, which
gives rise to a one to one map $\gamma_{s}$ from $[0,l(\gamma)]$ to $\Gamma$.
Let $\Gamma \ni x \to t(x) \in [0,l(\gamma)]$ be the inverse of $[0,l(\gamma)] \ni t \to \gamma_{s}(t) \in \Gamma$.
By the injectivity and continuity of $\gamma_{s}$, we have the crucial fact of continuity of its inverse. Namely, for every $x_{0} \in \Gamma$, we have
\begin{equation}
\label{eqq.4.20}
\lim\limits_{\substack{x \to x_{0} \\ x \in \Gamma}}|t(x)-t(x_{0})| = 0.
\end{equation}

Now we apply \eqref{eqq.4.19} with $\widetilde{\gamma}=\gamma_{s}$ and $[c,d]=[0,l(\gamma)]$. Let $\underline{E} \subset [0,l(\gamma)]$ be the set of all
$t_{0} \in [a,b]$ at which the second equality in \eqref{eqq.4.19} holds with $\widetilde{\gamma}$ replaced by $\gamma_{s}$ for every $t \in \underline{E}$. We put $\underline{\Gamma}:=\gamma_{s}(\underline{E})$.
Since $\mathcal{L}^{1}([0,l(\gamma)]\setminus \underline{E})=0$, by Remark \ref{Rem.1} we have $\mathcal{H}^{1}(\Gamma \setminus \underline{\Gamma})=0$.
Taking into account \eqref{eqq.4.20},  for each $\underline{x} \in \underline{\Gamma}$, we get (we use notation $t_{0}=t(\underline{x})$ and take into account change of variables property in the limit)
\begin{equation}
\begin{split}
\label{eqq.21}
&\lim\limits_{\substack{x_{1},x_{2} \to x \\ x_{1},x_{2} \in \Gamma}}\frac{l(\gamma_{x_{1},x_{2}})}{\operatorname{d}(x_{1},x_{2})} = \lim\limits_{\substack{x_{1},x_{2} \to x \\ x_{1},x_{2} \in \Gamma}}\frac{l(\gamma_{s}|_{[t(x_{1}),t(x_{2})]})}{\operatorname{d}(\gamma_{s}(t(x_{1})),\gamma_{s}(t(x_{2})))}\\
&= \lim\limits_{\varepsilon,\varepsilon' \to +0}\frac{l(\gamma_{s}|_{[t_{0}-\varepsilon,t_{0}+\varepsilon']})}{\operatorname{d}(\gamma_{s}(t_{0}-\varepsilon),\gamma_{s}(t_{0}+\varepsilon'))}=1.
\end{split}
\end{equation}

The lemma is proved.
\end{proof}

\subsection{Lipschitz maps}
In this subsection we gather some crucial concepts related to Lispchitz functions that will be very important in proving the main results of this paper.

\begin{Def}
\label{Def.3}
Given $L > 0$ and metric spaces $\operatorname{X}=(\operatorname{X},\operatorname{d})$, $\operatorname{Y}=(\operatorname{Y},\rho)$, a map $F:\operatorname{X} \to \operatorname{Y}$ is said to be \textit{$L$-Lipschitz}
provided that
\begin{equation}
\label{eqq.Lipschitz_map}
\rho(F(x),F(y)) \le L \operatorname{d}(x,y) \quad \hbox{for all} \quad x,y \in \operatorname{X}.
\end{equation}
We say that $F:\operatorname{X} \to \operatorname{Y}$ is \textit{Lipschitz} if it is $L$-Lipschitz for some $L > 0$.
Furthermore, by $L_{F}$ we denote the global Lipschitz constant of $F$, i.e., the least $L$ for which \eqref{eqq.Lipschitz_map} holds.
\end{Def}

The symbol $\operatorname{LIP}_{L}(\operatorname{X},\operatorname{Y})$
denotes the set of all $L$-Lipschitz maps from $\operatorname{X}$ to $\operatorname{Y}$.
We also put $\operatorname{LIP}(\operatorname{X},\operatorname{Y}):=\cup_{L > 0}\operatorname{LIP}_{L}(\operatorname{X},\operatorname{Y})$.
In what follows, we set $\operatorname{LIP}_{L}(\operatorname{X}):=\operatorname{LIP}_{L}(\operatorname{X},\mathbb{R})$,
$L > 0$ and $\operatorname{LIP}(\operatorname{X}):=\operatorname{LIP}(\operatorname{X},\mathbb{R})$ for brevity.

Let $\operatorname{K} \subset \operatorname{X}$ be a bonded nonempty set. By $\operatorname{LIP}^{\rm b}(\operatorname{K})$ we denote the linear space $\operatorname{LIP}(\operatorname{K})$ equipped with the $\sup$-norm, i.e.,
\begin{equation}
\label{eqq.sup_norm}
\|h\|_{\operatorname{LIP}^{\rm b}(\operatorname{K})}:=\sup\limits_{x \in \operatorname{K}}|h(x)|.
\end{equation}
Furthermore, by  $\operatorname{LIP}^{\rm str}(\operatorname{K})$ we mean the space $\operatorname{LIP}(\operatorname{K})$ equipped with the norm
\begin{equation}
\label{eqq.c_1_norm}
\|h\|_{\operatorname{LIP}^{\rm str }(\operatorname{K})}:=\|h\|_{\operatorname{LIP}^{\rm b}(\operatorname{K})}+L_{h}.
\end{equation}

\begin{Remark}
\label{Rem.Lip_Banach}
It is easy to see that $\operatorname{LIP}^{\rm b}(\operatorname{K})$ and $\operatorname{LIP}^{\rm str}(\operatorname{K})$ are Banach spaces.
\end{Remark}

Given a map $F \in \operatorname{LIP}(\operatorname{X},\operatorname{Y})$, we define
the \textit{local Lipschitz constant of $F$} by
\begin{equation}
\label{eqq.local_lip}
\operatorname{lip}F(x):=
\begin{cases}
\varlimsup\limits_{y \to x}\frac{\rho(F(x),F(y))}{\operatorname{d}(x,y)}, \quad \text{$x$ is an accumulation point};\\
0, \quad \text{otherwise}.
\end{cases}
\end{equation}

\begin{Remark}
\label{Rem.pointwise_Lipschitz}
Sometimes it will be convenient to deal with the \textit{pointwise Lipschitz constant.} More precisely, consider a Borel map $F:\operatorname{X} \to \operatorname{Y}$. For each $\underline{x} \in \operatorname{X}$, we put
\begin{equation}
\label{eqq.pointwise_Lipschitz}
\operatorname{Lip}F(\underline{x}):=\sup\limits_{y \neq \underline{x}}\frac{\rho(F(\underline{x}),F(y))}{\operatorname{d}(\underline{x},y)}.
\end{equation}
Clearly, given $L > 0$,  $F$ belongs to $\operatorname{LIP}_{L}(\operatorname{X},\operatorname{Y})$ if and only if $\operatorname{Lip}F(\underline{x}) \le L$ for all
$\underline{x} \in \operatorname{X}$ and
$$
L_{F}=\sup\limits_{\underline{x} \in \operatorname{X}}\operatorname{Lip}F(\underline{x}).
$$
\end{Remark}

The following lemma is obvious. We present the details for completeness.

\begin{Lm}
\label{Lm.1}
Let $\{h_{i}\} \subset \operatorname{LIP}_{L}(\operatorname{X})$ for some $L > 0$. Let $E \subset \operatorname{X}$ be a nonempty set with $\operatorname{diam}E < +\infty$, and let
a sequence $\{x_{i}\} \subset E$ be such that $M:=\sup_{i \in \mathbb{N}}|h_{i}(x_{i})| < +\infty$.
Then
$$
\sup\limits_{i \in \mathbb{N}}\sup\limits_{x \in E}|h_{i}(x)| \le M+L\operatorname{diam}E.
$$
\end{Lm}

\begin{proof}
Given $i \in \mathbb{N}$, it follows from \eqref{eqq.Lipschitz_map} that
$$
\sup\limits_{x \in E}|h_{i}(x)| \le |h_{i}(x_{i})|+\sup\limits_{x \in E}|h(x)-h(x_{i})| \le M+L\sup\limits_{x \in E}\operatorname{d}(x,x_{i}) \le M+L\operatorname{diam}E.
$$
Since $i \in \mathbb{N}$ was chosen arbitrarily, the lemma follows.
\end{proof}

We also recall the classical McShane--Whitney extension lemma (see \S 4.1 in \cite{HKST} for details).

\begin{Lm}
\label{Lm.McShane_Whitney}
Let $\operatorname{X}=(\operatorname{X},\operatorname{d})$ be a metric space and let $E \subset \operatorname{X}$ be a nonempty set. Given $L > 0$, for each $\widetilde{h} \in \operatorname{LIP}_{L}(E)$,
there exists $h \in \operatorname{LIP}_{L}(\operatorname{X})$ such that $h|_{E}=\widetilde{h}$.
\end{Lm}

\subsection{Rectifiable sets}
We recall the classical concept of rectifiability in metric spaces.
\begin{Def}
\label{Def_n_rectifiable}
Given a metric space $\operatorname{X}=(\operatorname{X},\operatorname{d})$ and a parameter $n \in \mathbb{N}$, an $\mathcal{H}^{n}$-measurable
set $E \subset \operatorname{X}$  is said to be \textit{$n$-rectifiable} if there exist a sequence $\{A_{i}\}_{i=1}^{\infty}$ of Borel
subsets of $\mathbb{R}^{n}$ and maps $f_{i} \in \operatorname{LIP}(A_{i},\operatorname{X})$, $i \in \mathbb{N}$, such that
$$
\mathcal{H}^{n}(E \setminus \bigcup\limits_{i=1}^{\infty}f_{i}(A_{i}))=0.
$$
\end{Def}

\begin{Remark}
\label{Rem.rectifiable_curve_is_a_rectifiable_set}
It is easy to show using the arc length parametrization that if $\gamma:[a,b] \to \operatorname{X}$ is a rectifiable curve, then
the set $\Gamma$ is $1$-rectifiable in the sense of Definition \ref{Def_n_rectifiable}.
\end{Remark}

The crucial for the sequel fact concerning rectifiable sets is as follows (see Theorem 9 in \cite{Kirk} and the remark after the theorem).
\begin{Prop}
\label{Prop_density}
Let $\operatorname{X}=(\operatorname{X},\operatorname{d})$ be  a complete metric space. If $E \subset \operatorname{X}$  is $n$-rectifiable, then
\begin{equation}
\lim\limits_{r \to 0}\frac{\mathcal{H}^{n}(B_{r}(x) \cap E)}{(2r)^{n}}=1 \qquad \hbox{for $\mathcal{H}^{n}$-a.e. $x \in E$}.
\end{equation}
\end{Prop}

Now we prove a technical assertion concerning local Lipschitz constants, which will be quite useful in \S 3 below.
\begin{Lm}
\label{Lm.local_lip_const}
Let $\operatorname{X}:=(\operatorname{X},\operatorname{d})$ be a metric space and let
$\gamma \in C([a,b],\operatorname{X})$ be a simple rectifiable curve of positive length. Let $E \subset \Gamma$ be a closed set with $\mathcal{H}^{1}(E) > 0$.
Let $h \in \operatorname{LIP}(\Gamma)$ be such that, for $\mathcal{H}^{1}$-a.e. $x \in E$, there exists
$$
\lim\limits_{\substack{y \to x \\ y \in E}}\frac{|h(y)-h(x)|}{\operatorname{d}(x,y)} \in [0,+\infty].
$$
Then, for $\mathcal{H}^{1}$-a.e. $x \in E$, there exists
\begin{equation}
\label{eqq.local_lip_const}
\lim\limits_{\substack{y \to x \\ y \in \Gamma}}\frac{|h(y)-h(x)|}{\operatorname{d}(x,y)} = \lim\limits_{\substack{y \to x \\ y \in E}}\frac{|h(y)-h(x)|}{\operatorname{d}(x,y)}.
\end{equation}
\end{Lm}

\begin{proof}
We fix $h \in \operatorname{LIP}_{L}(\Gamma)$ for some $L > 0$. Since $E \subset \Gamma$, it is clear that
$$
\lim\limits_{\substack{y \to x \\ y \in E}}\frac{|h(x)-h(y)|}{\operatorname{d}(x,y)} = \varliminf\limits_{\substack{y \to x \\ y \in E}}\frac{|h(x)-h(y)|}{\operatorname{d}(x,y)} \le \varliminf\limits_{\substack{y \to x \\ y \in \Gamma}}\frac{|h(x)-h(y)|}{\operatorname{d}(x,y)} \quad \hbox{for each $x \in E$}.
$$
In order to establish \eqref{eqq.local_lip_const} it is sufficient to show that
\begin{equation}
\label{eqq.local_lip''}
\lim\limits_{\substack{y \to x \\ y \in E}}\frac{|h(y)-h(x)|}{\operatorname{d}(x,y)} \geq \varlimsup\limits_{\substack{y \to x \\ y \in \Gamma}}\frac{|h(y)-h(x)|}{\operatorname{d}(x,y)} \quad \hbox{for $\mathcal{H}^{1}$-a.e. $x \in E$}.
\end{equation}

For each $x \in E$, we put
$$
\operatorname{Por}_{E}(x,r):=\sup\{\sigma > 0: B_{\sigma}(y) \cap E = \emptyset \text{ and } B_{\sigma}(y) \subset B_{r}(x) \text{ for some } y \in \Gamma \setminus E\}.
$$
We claim that
\begin{equation}
\label{eqq.porosity}
\operatorname{Por}_{E}(x,r) = o(r), r \to 0 \qquad \hbox{for $\mathcal{H}^{1}$-a.e. $x \in E$}.
\end{equation}
Indeed, since $\Gamma$ is a path connected set, it is easy to see that $\mathcal{H}^{1}(B_{\sigma}(y) \cap \Gamma) \geq \sigma$ for each $y \in \Gamma$ and for every small
enough $\sigma > 0$. On the other hand, by Remark \ref{Rem.rectifiable_curve_is_a_rectifiable_set}
the set $\Gamma=\gamma([a,b])$ is $1$-rectifiable. Hence, so is
the set $E$. An application of Proposition \ref{Prop_density} gives
$$
\lim\limits_{r \to 0}\frac{\mathcal{H}^{1}(B_{r}(x) \cap E)}{2r} = 1 \qquad \hbox{for $\mathcal{H}^{1}$-a.e. point $x \in E$}.
$$
This immediately leads to \eqref{eqq.porosity}.

We fix $x \in E$ for which
$\operatorname{Por}_{E}(x,r) = o(r)$, $r \to 0$. Now, given $y \in \Gamma$, let $\underline{y} \in E$ be an arbitrary metric projection of $y$ to $E \cap \overline{B}_{\operatorname{d}(x,y)}(x)$.
Let $\{y_{n}\} \subset \Gamma$ be such that $\lim_{n \to \infty}y_{n} = x$, $y_{n} \neq x$ for all $n \in \mathbb{N}$, and
$$
\varlimsup\limits_{\substack{y \to x \\ y \in \Gamma}}\frac{|h(x)-h(y)|}{\operatorname{d}(x,y)}=\lim\limits_{n \to \infty}\frac{|h(x)-h(y_{n})|}{\operatorname{d}(x,y_{n})}.
$$
Hence,
\begin{equation}
\notag
\begin{split}
&|h(x)-h(y_{n})| \le |h(x)-h(\underline{y_{n}})| + |h(y_{n})-h(\underline{y_{n}})| \le |h(x)-h(\underline{y_{n}})|+L \operatorname{d}(y_{n},\underline{y_{n}})\\
&\le |h(x)-h(\underline{y_{n}})|+2L\operatorname{Por}_{E}(x,\operatorname{d}(x,y_{n})) = |h(x)-h(\underline{y_{n}})|+o(\operatorname{d}(x,y_{n})).
\end{split}
\end{equation}
Combining this inequality with the fact that $\operatorname{d}(x,\underline{y_{n}})=\operatorname{d}(x,y_{n})+o(\operatorname{d}(x,y_{n})), n \to \infty$, we get
\begin{equation}
\notag
\frac{|h(x)-h(y_{n})|}{\operatorname{d}(x,y_{n})} \le \frac{|h(x)-h(\underline{y_{n}})|}{\operatorname{d}(x,\underline{y_{n}})}\frac{\operatorname{d}(x,\underline{y_{n}})}{\operatorname{d}(x,y_{n})}+o(1) \le
\frac{|h(x)-h(\underline{y_{n}})|}{\operatorname{d}(x,\underline{y_{n}})}\Bigl(1+o(1)\Bigr)+o(1).
\end{equation}
As a result, we derive
\begin{equation}
\varlimsup\limits_{n \to \infty}\frac{|h(x)-h(y_{n})|}{\operatorname{d}(x,y_{n})} \le \varlimsup\limits_{n \to \infty}\frac{|h(x)-h(\underline{y_{n}})|}{\operatorname{d}(x,\underline{y_{n}})}
\le \lim\limits_{\substack{y \to x \\ y \in E}}\frac{|h(x)-h(y)|}{\operatorname{d}(x,y)}.
\end{equation}
This establishes the opposite inequality in \eqref{eqq.local_lip_const} and completes the proof.
\end{proof}

The crucial for the sequel fact about curves in metric spaces is given by the following classical result (see Theorem 4.4.5 in \cite{AmbTil}, for instance).

\begin{Prop}
\label{Prop.rectifiability}
Let $\gamma \in C([a,b],\operatorname{X})$ be such that $\mathcal{H}^{1}(\Gamma) < +\infty$. Then there exists an at most countable family $\{\gamma^{i}\}_{i \in I} \subset C([0,1],\operatorname{X})$ such that
(we identify $I$ with either $\{1,...,N\}$ for some $N \in \mathbb{N}$ or with $\mathbb{N}$):

(1) For each $i \in I$ the map $\gamma^{i}$ belongs to $\operatorname{LIP}([0,1],\operatorname{X})$;

(2) For each $i \in I$ the map $\gamma^{i}$ is injective;

(3) $\mathcal{H}^{1}(\Gamma \setminus \bigcup\limits_{i \in I}\Gamma^{i}) = 0$, where $\Gamma^{i}:=\gamma^{i}([0,1])$, $i \in I$;

(4) For each $i \in I$ with $i \geq 2$ the intersection $\Gamma^{i} \bigcap \bigcup_{j=1}^{i-1}\Gamma^{j}$ consists of a single point $\underline{x}^{i}$.

\end{Prop}

\subsection{Jacobians}
We fix a  metric space $\operatorname{X}=(\operatorname{X},\operatorname{d})$ and a curve $\gamma \in C([a,b],\operatorname{X})$.
We put $\Gamma:=\gamma([a,b])$, as usual.
Now, following the approach introduced in \cite{Bate}, we define the Jacobain of a given Lipschitz function with respect to $\Gamma$.

\begin{Def}
\label{Def.Jacobian}
Let $x \in \Gamma$ and suppose that there are sets $E \subset \Gamma$, $S \subset \mathbb{R}$ and a biLipschitz map $\overline{\gamma}: S \to E$ such that $x$ is a density point of $E$.
Suppose that $|\dot{\overline{\gamma}}|(\overline{\gamma}^{-1}(x)) > 0$. Suppose $h \in \operatorname{LIP}(\operatorname{X})$ is such that $(h \circ \overline{\gamma})'(\overline{\gamma}^{-1}(x))$
exists. Then we define the Jacobian of $h$ at $x$ with respect to $\Gamma$ as
\begin{equation}
\mathcal{J}_{\Gamma}h(x):= \frac{|(h \circ \overline{\gamma})'(\overline{\gamma}^{-1}(x))|}{|\dot{\overline{\gamma}}(\overline{\gamma}^{-1}(x))|}.
\end{equation}
\end{Def}

\begin{Remark}
\label{Rem.correct_Jacobian}
It is not difficult to show that Definition \ref{Def.Jacobian} is well posed, i.e. it does not depend on the choice of $\overline{\gamma}$.
Note that the Jacobian makes sense $\mathcal{H}^{1}$-a.e. on $\Gamma$.
Furthermore, it is mentioned in \cite{Bate} that this definition is consistent with an alternative definition of the Jacobian given in \cite{AmbKirk}.
\end{Remark}

The following simple assertion will be useful in \S 5.

\begin{Lm}
\label{Lm.Jacobian_Lip_constant}
For each $h \in \operatorname{LIP}(\operatorname{X})$, for $\mathcal{H}^{1}$-a.e. $x \in \Gamma$ there exists the limit
$$
\lim_{y \to x}\frac{|h(y)-h(x)|}{\operatorname{d}(x,y)} = \operatorname{lip}h(x) = \mathcal{J}_{\Gamma}h(x).
$$
\end{Lm}

\begin{proof} We recall Proposition \ref{Prop.rectifiability}. Since $I$ is an at most countable index set, we have
$$
\mathcal{H}^{1}(\{\underline{x}^{i}:i \in I\})=0.
$$
Hence, it is sufficient to prove the lemma in the particular case when $\gamma$ is simple.

Taking into account Lemma \ref{Lm.inner_and_outer_metric}, we conclude that if $\gamma_{s}$ is the arc length
parametrization of the curve $\gamma$, then for $\mathcal{H}^{1}$-a.e. point $x \in \Gamma$ there is a small enough $\delta(x) > 0$ such that $\gamma_{s}|_{[\gamma_{s}^{-1}(x)-\delta(x),\gamma_{s}^{-1}(x)+\delta(x)]}$ is
a biLipschitz map. Since the composition $h \circ \gamma_{s}$ is Lipschitz, it is differentiable $\mathcal{L}^{1}$-a.e. on $[0,l(\gamma)]$ by the Rademacher theorem. Hence,
by Remark \ref{Rem.1} for $\mathcal{H}^{1}$-a.e. $x \in \Gamma$, there exists
$$
\lim\limits_{y \to x}\frac{|h(y)-h(x)|}{|\gamma_{s}^{-1}(x)-\gamma_{s}^{1}(y)|} = |(h\circ\gamma_{s})'(\gamma_{s}^{-1}(x))|.
$$
It remains to combine this observation with Definition \ref{Def.Jacobian}, Lemma \ref{Lm.inner_and_outer_metric} and Remark \ref{Rem.correct_Jacobian}.
\end{proof}

\subsection{Area-type formulas}
We recall the general area formula (see Theorem 2.8 in \cite{Bate}).
\begin{Prop}
\label{Prop.area_formula}
Let $\operatorname{Y}=(\operatorname{Y},\rho)$ be  a nonempty metric space and let $\gamma \in C([a,b],\operatorname{Y})$. For every function
$h \in \operatorname{LIP}(\operatorname{Y})$, and for every Borel function $\theta:\Gamma \to [0,+\infty]$,
\begin{equation}
\int\limits_{\Gamma}\theta(y)\mathcal{J}_{\Gamma}h(y)\,d\mathcal{H}^{1}(y) = \int\limits_{h(\Gamma)}\sum\limits_{y \in h^{-1}(t)}\theta(y)\,dt.
\end{equation}
\end{Prop}

We also need an integral representation for the variation (see Theorem 2.10.13 in \cite{Federer}).
\begin{Prop}
\label{Prop.integral_repres_for_variation}
Given a metric space $\operatorname{Y}=(\operatorname{Y},\rho)$, for each map $g \in C([a,b],\operatorname{Y})$,
\begin{equation}
V_{g}([a,b])=\int\limits_{g([a,b])}\#g^{-1}(y)\,d\mathcal{H}^{1}(y).
\end{equation}
\end{Prop}

\subsection{Absolutely continuous and Sobolev curves}
We fix a  metric space $\operatorname{X}=(\operatorname{X},\operatorname{d})$.
We start with a characterization of continuity via Lipschitz post-compositions.
\begin{Lm}
\label{Lm.continuity}
Let $\gamma: [a,b] \to \operatorname{X}$ be such that $h \circ \gamma \in C([a,b])$ for every $h \in \operatorname{LIP}_{1}(\operatorname{X})$.
Then the map $\gamma$ is continuous.
\end{Lm}

\begin{proof}
Given $x \in \operatorname{X}$, we set $h_{x}(y):=\operatorname{d}(x,y)$, $y \in \operatorname{X}$.
Note that, given a point $t_{0} \in [a,b]$, the map $\gamma$ is continuous at $t_{0}$ if and only if $h_{x_{0}} \circ \gamma$ is continuous at $t_{0}$, where $x_{0}:=\gamma(t_{0})$.
Since $h_{x} \in \operatorname{LIP}_{1}(\operatorname{X})$ for every $x \in \operatorname{X}$, the lemma follows.
\end{proof}

Now we establish a similar property in the context of the Borel regularity.
\begin{Lm}
\label{Lm.Borel_regularity}
Assume that $\operatorname{X}$ is separable. Let $\gamma: [a,b] \to \operatorname{X}$ be such that $h \circ \gamma \in \mathfrak{B}([a,b])$ for every $h \in \operatorname{LIP}(\operatorname{X})$.
Then the map $\gamma$ is Borel.
\end{Lm}

\begin{proof}
Using the Fr\'echet imbedding theorem (see \S 4.1 in \cite{HKST}) we may assume that $\operatorname{X}$ is a subset of $l_{\infty}$.
Clearly, the restriction of every functional $h \in l_{\infty}^{\ast}$ to $\operatorname{X}$ is a Lipschitz function. Hence, $\gamma$ is weakly Borel.
Since $\operatorname{X}$ is separable, by the Pettis measurability theorem (see Corollary 3.1.2 in \cite{HKST}) the map $\gamma$ is Borel.
\end{proof}

Following \cite{AGS} we introduce the concept of $p$-absolutely continuous curves.
\begin{Def}
\label{Def.4}
Given a metric space $\operatorname{Y}=(\operatorname{Y},\rho)$, we say that a map $\gamma:[a,b] \to \operatorname{Y}$ is \textit{absolutely continuous} if there is a nonnegative $g \in \widetilde{L}_{1}([a,b])$ such that
\begin{equation}
\label{eqq.absolute_continuity}
\operatorname{d}(\gamma(t_{1}),\gamma(t_{2})) \le \int\limits_{t_{1}}^{t_{2}}g(t)\,dt \quad \hbox{for all} \quad a \le t_{1} \le t_{2} \le b.
\end{equation}
For $p \in [1,\infty]$, by $AC_{p}([a,b],\operatorname{Y})$ we denote the space of all $p$-absolutely continuous curves, i.e.
those absolutely continuous curves for which we
can find $g$ as above in the space $\widetilde{L}_{p}([a,b])$.
\end{Def}

In the case $\operatorname{Y}=(\mathbb{R},|\cdot|)$, $AC_{p}([a,b])$ denotes the corresponding space of all $p$-absolutely continuous real-valued functions.

\begin{Remark}
\label{Rem.Lipschitz_case}
Note that  the class $AC_{\infty}([a,b],\operatorname{Y})$ can be identified with the class $\operatorname{LIP}([a,b],\operatorname{Y})$.
\end{Remark}

We present the following elementary characterization of $AC_{p}$-curves.
The sufficiency part of the proof is essentially based on the ideas from the proof of  Theorem 1.1.2 in \cite{AGS}.
\begin{Lm}
\label{Lm.elementary_characterization}
Given $p \in [1,\infty]$, a map $\gamma: [a,b] \to \operatorname{X}$ belongs to the space  $AC_{p}([a,b],\operatorname{X})$ if and only if there exists a nonnegative Borel function $g \in \widetilde{L}_{p}([a,b])$ such that for every $h \in \operatorname{LIP}(\operatorname{X})$ the postcomposition $h \circ \gamma$
belongs to $AC_{p}([a,b])$ and
\begin{equation}
\label{eqq.AC_ineq'}
\Bigl|\frac{d(h \circ \gamma)}{dt}(t)\Bigr| \le \operatorname{lip}h(\gamma(t)) g(t) \quad \hbox{for $\mathcal{L}^{1}$-a.e. $t \in [a,b]$}.
\end{equation}

Furthermore, if $\gamma \in AC_{p}([a,b],\operatorname{X})$, then for $\mathcal{L}^{1}$-a.e.\ $t \in [a,b]$ there
exists the limit
\begin{equation}
\label{eqq.metric_speed}
\lim\limits_{h \to 0}\frac{\operatorname{d}(\gamma(t),\gamma(t+h))}{|h|}=|\dot{\gamma}(t)|=\operatorname{lip}\gamma(t),
\end{equation}
it defines a nonnegative function  $|\dot{\gamma}| \in \widetilde{L}_{p}([a,b])$ called the metric speed of $\gamma$, which is the least - in the $\mathcal{L}^{1}$-a.e.\ sense - function $g$ for which \eqref{eqq.absolute_continuity}
holds. Moreover, there exists a countable family $\{h_{n}\} \subset \operatorname{LIP}_{1}(\operatorname{X})$ such that
\begin{equation}
\label{eqq.metric_speed_representation}
|\dot{\gamma}(t)|=\sup\limits_{n \in \mathbb{N}}|(h_{n} \circ \gamma)'(t)| \quad \hbox{for $\mathcal{L}^{1}$-a.e. $t \in [a,b]$}.
\end{equation}
\end{Lm}

\begin{proof}
Let $\gamma \in AC_{p}([a,b],\operatorname{X})$. By Definitions \ref{Def.3} and \ref{Def.4}, it follows that for each $h \in \operatorname{LIP}_{L}(\operatorname{X})$,
and all $a \le t_{1} \le t_{2} \le b$,
$$
|h\circ\gamma(t_{1}) - h\circ \gamma(t_{2})| \le L \operatorname{d}(\gamma(t_{1}),\gamma(t_{2})) \le L\int\limits_{t_{1}}^{t_{2}}g(\tau)\,d\tau.
$$
Hence, $h \circ \gamma \in AC_{p}([a,b])$. Since $AC_{p}([a,b]) \subset AC_{1}([a,b])$, given $h \in \operatorname{LIP}(\operatorname{X})$, the function $h \circ \gamma$
is differentiable $\mathcal{L}^{1}$-a.e. in the classical sense. Furthermore, using the Lebesgue differentiation theorem and keeping in mind \eqref{eqq.local_lip} we deduce
\eqref{eqq.AC_ineq'}.

Conversely, assume that there exists a nonnegative Borel function $g \in L_{p}([a,b])$ such that for every $h \in \operatorname{LIP}(\operatorname{X})$ the postcomposition $h \circ \gamma$
belongs to $AC_{p}([a,b])$ and \eqref{eqq.AC_ineq'} holds. By Lemma \ref{Lm.continuity} the map $\gamma$ is in fact belongs to $C([a,b],\operatorname{X})$.
Hence, the metric space $(\Gamma,\operatorname{d}|_{\Gamma})$ is complete and separable. Let $\{x_{n}\} \subset \Gamma$ be an arbitrary countable dense subset of $\Gamma$.
Given $n \in \mathbb{N}$, we put $h_{n}(x):=\operatorname{d}(x_{n},x)$, $x \in \operatorname{X}$. Clearly, $\{h_{n}\} \subset \operatorname{LIP}_{1}(\operatorname{X})$ and, consequently,
$\operatorname{lip}h_{n}(x) \le 1$ for all $x \in \operatorname{X}$. As a result,
\begin{equation}
\label{eqq.AGS_estimate}
\operatorname{d}(\gamma(t_{1}),\gamma(t_{2})) = \sup\limits_{n \in \mathbb{N}} |h_{n}\circ\gamma(t_{1}) - h_{n}\circ \gamma(t_{2})| \le \int\limits_{t_{1}}^{t_{2}}\sup\limits_{n \in \mathbb{N}}|(h_{n} \circ \gamma)'(t)|\,dt \le
\int\limits_{t_{1}}^{t_{2}}g(t)\,dt.
\end{equation}
This chain of estimates clearly implies that $m(t):=\sup_{n \in \mathbb{N}}|(h_{n} \circ \gamma)'(t)| \le g(t)$ for $\mathcal{L}^{1}$-a.e. $t \in [a,b]$.
In particular, the curve $\gamma$ belongs to the class $AC_{p}([a,b],\operatorname{X})$.

Furthermore, by \eqref{eqq.AGS_estimate} we have $m \in \widetilde{L}_{p}([a,b])$ and $m$ is the least - in the $\mathcal{L}^{1}$-a.e. sense - function $g$ for which \eqref{eqq.absolute_continuity}
holds. Using \eqref{eqq.AGS_estimate} again, we deduce that
\begin{equation}
\label{eqq.right}
\varlimsup\limits_{t' \to t}\frac{\operatorname{d}(\gamma(t),\gamma(t'))}{|t-t'|} \le m(t) \quad \hbox{for $\mathcal{L}^{1}$-a.e. $t \in [a,b]$}.
\end{equation}
At the same time, for every $n \in \mathbb{N}$,
\begin{equation}
\label{eqq.left}
\varliminf\limits_{t' \to t}\frac{\operatorname{d}(\gamma(t),\gamma(t'))}{|t-t'|} \geq \varliminf\limits_{t' \to t}\frac{|h_{n}\circ\gamma(t)-h_{n}\circ\gamma(t')|}{|t-t'|} = |(h_{n} \circ \gamma)'(t)|  \quad \hbox{for $\mathcal{L}^{1}$-a.e. $t \in [a,b]$}.
\end{equation}
Taking the supremum over $n \in \mathbb{N}$ in \eqref{eqq.left}, we obtain the inequality opposite to \eqref{eqq.right}.
As a result, we arrive at \eqref{eqq.metric_speed} and \eqref{eqq.metric_speed_representation}.

The lemma is proved.
\end{proof}

There is an equivalent definition of absolutely continuous metric-valued curves which is closer in spirit to the classical definition of absolutely continuous functions \cite{Duda}.
Furthermore, there exists the Bahach--Zarecki-type characterization of absolutely continuous curves given in \cite{Duda} (see also \cite{DZ} for a similar result).
We summarize these facts in the following assertion.

\begin{Prop}
\label{Prop.BZ}
Let $\gamma:[a,b] \to \operatorname{X}$ be a map. The following properties are equivalent:

\begin{itemize}

\item[\((AC1)\)] $\gamma$ belongs to the class $AC_{1}([a,b],\operatorname{X})$;

\item[\((AC2)\)] for each $\varepsilon > 0$ there is $\delta(\varepsilon) > 0$ such that $\sum_{i=1}^{N}\operatorname{d}(\gamma(a_{i}),\gamma(b_{i})) < \varepsilon$ for any finite family $\{[a_{i},b_{i}]\}_{i=1}^{N}$ of noneoverlapping intervals in $[a,b]$ with $\sum_{i=1}^{N}|a_{i}-b_{i}| < \delta(\varepsilon)$;

\item[\((AC3)\)] $\gamma$ belongs to the intersection $C([a,b],\operatorname{X}) \cap BV([a,b],\operatorname{X})$ and
$\gamma$ satisfies the $N$-Luzin property, i.e., given $E \subset [a,b]$, $\mathcal{H}^{1}(\gamma(E))=0$ whenever $\mathcal{L}^{1}(E)=0$.

\end{itemize}

\end{Prop}

There are several equivalent approaches to the concept of Sobolev curves. We refer the reader to the recent
papers \cite{GT1,GT2}. We prefer the approach closer to Definition \ref{Def.4}.

\begin{Def}
\label{Def.Sobolev_curve}
Assume that $(\operatorname{Y},\rho)$ is a complete separable metric space. Given $p \in (1,\infty)$, we say that the equivalence class $[\gamma]$
of a Borel map $\gamma:[a,b] \to \operatorname{Y}$ belongs
to the Sobolev $W_{p}^{1}([a,b],\operatorname{Y})$-space if $[\gamma] \in L_{p}([a,b],\operatorname{Y})$ and there exists a nonnegative function $G \in \widetilde{L}_{p}([a,b])$ with a Borel
set $E \subset [a,b]$ with $\mathcal{L}^{1}([a,b]\setminus E)=0$ such that
\begin{equation}
\label{eqq.Sobolev_definition}
\operatorname{d}(\gamma(t_{1}),\gamma(t_{2})) \le \int\limits_{t_{1}}^{t_{2}}G(t)\,dt \qquad \hbox{for all   $t_{1}, t_{2} \in E$ with $t_{1} \le t_{2}$}.
\end{equation}
\end{Def}

\begin{Remark}
In the case $(\operatorname{Y},\rho)=(\mathbb{R},|\cdot|)$, we write $W_{p}^{1}([a,b])$ instead of $W_{p}^{1}([a,b],\mathbb{R})$. One can easily show that,
in this case, $W_{p}^{1}([a,b])$ can be naturally identified with the classical Sobolev space $W_{p}^{1}((a,b))$ on the open interval $(a,b)$.
\end{Remark}

Keeping in mind \eqref{eqq.equivalence_class} and the equivalence of different approaches to Sobolev classes proved in Corollary 3.10 of \cite{GT2} and Theorem 2.11 from \cite{GT1}, we
obtain
the following characterization.
\begin{Prop}
\label{Prop.Sobolev_characterization}
Let $(\operatorname{X},\operatorname{d})$ be a complete separable metric space and $p \in (1,\infty)$. For each $\gamma \in \mathfrak{B}([a,b],\operatorname{X})$,
the following conditions are equivalent:

\begin{itemize}

\item[\((W1)\)] the class $[\gamma]$ belongs to $W_{p}^{1}([a,b],\operatorname{X})$;

\item[\((W2)\)] the intersection $[\gamma]  \cap AC_{p}([a,b],\operatorname{X})$ is not empty;

\item[\((W3)\)] for every $h \in \operatorname{LIP}(\operatorname{X})$, the class $[h \circ \gamma]$ belongs to $W_{p}^{1}([a,b])$ and there is
a nonnegative Borel function $G \in \widetilde{L}_{p}([a,b])$, whose equivalence class $[G]$ is called a $p$-weak upper gradient of $\gamma$, such that
\begin{equation}
\label{eqq.Sobolev_ineq'}
\bigl|\partial [h \circ \gamma](t)\bigr| \le [\operatorname{lip}(h \circ \gamma)](t) [G](t) \quad \hbox{for $\mathcal{L}^{1}$-a.e. $t \in [a,b]$},
\end{equation}
where $\partial [h \circ \gamma]$ is the usual Sobolev distributional derivative of the class $[h \circ \gamma]$.

\end{itemize}

Furthermore, if $[\gamma] \in W_{p}^{1}([a,b],\operatorname{X})$, then the family $\{[\frac{\operatorname{d}(\gamma(\cdot),\gamma(\cdot+h))}{h}]\}_{h \in \mathbb{R} \setminus \{0\}}$ (we formally put $\frac{\operatorname{d}(\gamma(t),\gamma(t+h))}{h}=0$ if $t+h \notin [a,b]$)
has a strong limit $|\partial [\gamma]|$ - called distributional derivative of
$[\gamma]$ - in $L_{p}([a,b])$ as $h \to 0$. Moreover, if $\overline{\gamma}$ is a unique continuous representative of $[\gamma]$, then
$|\partial [\gamma] (t)| = |\dot{\overline{\gamma}}(t)|$ for $\mathcal{L}^{1}$-a.e. $t \in [a,b]$.

\end{Prop}

\section{Tracking $\mathcal{H}^{1}$-measure via post-compositions with Lipschitz functions.}

First of all, we start with a rather simple observation from the classical infinite dimensional analysis, which will be useful below.
We prefer to present a direct and explicit proof instead of using Baire-type arguments. Our proof has the same flavor as the construction
of a continuous function whose Fourier series diverges at some point.

Given a Banach space $(\operatorname{E},\|\cdot\|)$, we say that a nonnegative functional $p:\operatorname{E} \to [0,+\infty)$ is \textit{countably subadditive} if
$$
p(\sum\limits_{i=1}^{\infty}z_{i}) \le \sum_{i=1}^{\infty}p(z_{i})
$$
for any unconditionally convergent series $\sum_{i=1}^{\infty}z_{i}$ in $\operatorname{E}$.

\begin{Prop}
\label{Prop.Banach_Steinhaus_substitution}
Let $\operatorname{E}:=(\operatorname{E},\|\cdot\|)$ be a Banach space. Let $\{p_{m}\}_{m=1}^{\infty}$ be a sequence of countably subadditive nonnegative positively homogeneous functionals on $\operatorname{E}$ such that
$$
\operatorname{M}:=\sup\limits_{m \in \mathbb{N}}\sup\limits_{\|z\| \le 1}p_{m}(z) = +\infty.
$$
Then there exists $\underline{z} \in \operatorname{E}$, $\|\underline{z}\| \le 1$, such that $\sup_{m \in \mathbb{N}}p_{m}(\underline{z}) = +\infty$.
\end{Prop}

\begin{proof}
Without loss of generality we may assume that there is a sequence $\{z_{m}\} \subset \operatorname{E}$ such that $\|z_{m}\| \le 1$  and
$p_{m}(z_{m}) > m$ for all $m \in \mathbb{N}$. Furthermore, replacing $p_{m}$ by $\widetilde{p}_{m}:=\max_{1 \le k \le m}p_{k}$ if necessary,
we may assume that $p_{j+1} \geq p_{j}$ (in the pointwise sense) for all $j \in \mathbb{N}$.

We put $\alpha_{1}=1/2$ and $m_{1}=1$. Arguing by induction it is easy to
built a strictly decreasing sequence $\{\alpha_{j}\} \subset (0,+\infty)$ and a strictly increasing sequence $\{m_{j}\} \subset \mathbb{N}$
such that
$$
\max\{\alpha_{j+1},\alpha_{j+1}p_{m_{j}}(z_{m_{j+1}})\} \le 2^{-j} \quad \hbox{for every} \quad j \in \mathbb{N},
$$
and, furthermore,
\begin{equation}
\notag
\alpha_{j+1}p_{m_{j+1}}(z_{m_{j+1}}) \geq 3\max\Bigl\{j,\sum\limits_{i=1}^{j}\alpha_{i}p_{m_{j}}(z_{m_{i}})\Bigr\} \quad \hbox{for every} \quad j \in \mathbb{N}.
\end{equation}

We put $\underline{z}:=\sum_{i=1}^{\infty}\alpha_{i}z_{m_{i}}.$
Using the subadditivity and positive homogeneity of the functionals $p_{m}$, $m \in \mathbb{N}$, we obtain, for each $j \in \mathbb{N}$,
\begin{equation}
\notag
\begin{split}
&\sup_{m \in \mathbb{N}}p_{m}(\underline{z}) \geq p_{m_{j+1}}(\sum_{i=1}^{\infty}\alpha_{i}z_{m_{i}}) \geq \alpha_{j+1}p_{m_{j+1}}(z_{m_{j+1}}) - \sum\limits_{i=1}^{j}\alpha_{i}p_{m_{j+1}}(z_{m_{i}}) -
\sum\limits_{i=j+2}^{\infty}\alpha_{i}p_{m_{j+1}}(z_{m_{i}})\\
&\geq 2j - \sum\limits_{i=j+2}^{\infty}\alpha_{i}p_{m_{i-1}}(z_{m_{i}}) \geq j.
\end{split}
\end{equation}
Since $j$ can be chosen arbitrarily large, the claim follows.
\end{proof}

\begin{Lm}
\label{Lm.Banah_Steinhauz}
Let $\gamma:[a,b] \to \operatorname{X}$ and $\Gamma:=\gamma([a,b])$. Let $\{h_{m}\}_{m=1}^{\infty} \subset \operatorname{LIP}^{\rm str}(\Gamma)$ be such that
$\sup_{m \in \mathbb{N}}\|h_{m}\|_{\operatorname{LIP}^{\rm str}(\Gamma)} \le L$ for some $L \in (0,+\infty)$ and $\sup_{m \in \mathbb{N}} V_{h_{m} \circ \gamma}([a,b]) = +\infty.$
Then there is $\underline{h} \in \operatorname{LIP}^{\rm str}(\Gamma) \cap \operatorname{LIP}(\operatorname{X})$ with $\|\underline{h}\|_{\operatorname{LIP}^{\rm str}(\Gamma)} \le L$ such that $\underline{h} \circ \gamma \notin BV([a,b])$.
\end{Lm}

\begin{proof}
Replacing $h_{m}$ by $h_{m}/L$ for every $m \in \mathbb{N}$, we may assume without loss of generality that $L=1$.
Passing to a subsequence if necessary, we may assume that $V_{h_{m} \circ \gamma}([a,b]) > m$ for each $m \in \mathbb{N}$.
Given $m \in \mathbb{N}$, we fix $\operatorname{T}_{m}:=\{t_{m,k}\}_{k=1}^{N_{m}} \subset \mathcal{T}([a,b])$ such that $V_{h_{m} \circ \gamma}(\operatorname{T}_{m}) > m$.
For each $m \in \mathbb{N}$, we put
$$
p_{m}(h):=V_{h \circ \gamma}(\operatorname{T}_{m}), \quad h \in \operatorname{LIP}^{\rm str}(\Gamma).
$$
Furthermore, replacing $\operatorname{T}_{m}$ by $\widetilde{\operatorname{T}}_{m}:=\cup_{k=1}^{m}\operatorname{T}_{k}$ if necessary, we may assume that $\operatorname{T}_{m+1} \succeq \operatorname{T}_{m}$
for all $m \in \mathbb{N}$. Combining Remarks \ref{Rem.monotonicity_variation}, \ref{Rem.Lip_Banach} with Lemma \ref{Lm.McShane_Whitney} and Proposition \ref{Prop.Banach_Steinhaus_substitution} we conclude.
\end{proof}

\begin{Remark}
\label{Rem.Banach_Steinhaus}
If we are not interested in a concrete construction of $h$ in Lemma \ref{Lm.Banah_Steinhauz} and merely care about its existence, we can essentially shrink the proof.
Indeed, it is sufficient to use Banach--Steinhaus-type arguments.
\end{Remark}

The main result of this section reads as follows.

\begin{Th}
\label{Th.1}
Let $\gamma \in C([a,b],\operatorname{X})$ be such that $h \circ \gamma \in BV([a,b])$ for every $h \in \operatorname{LIP}(\operatorname{X})$. Then $\mathcal{H}^{1}(\Gamma) < +\infty$.
\end{Th}

\begin{proof} Assume on the contrary  that $\mathcal{H}^{1}(\Gamma)=+\infty$.
We split the proof into several steps.

\textit{Step 1.}
Given $m \in \mathbb{N}$, there is $\delta_{m} \in (0,\frac{1}{m})$ and a family $\mathcal{U}_{m}:=\{U_{m,i}\}_{i=1}^{\infty}$ of sets with $\epsilon_{m,i}:=\operatorname{diam}U_{m,i}$
such that:

(1) $\mathcal{H}^{1}_{10\delta_{m}}(\Gamma) > 20m$;

(2) $\epsilon_{m,i} < \delta_{m}$ for all $i \in \mathbb{N}$;

(3) $U_{m,i} \cap \Gamma \neq \emptyset$ for all $i \in \mathbb{N}$;

(4) $\Gamma \subset \cup_{i=1}^{\infty}U_{m,i}$.

For each $i \in \mathbb{N}$, we take a point $z_{m,i} \in U_{m,i} \cap \Gamma$ and consider the family $\{B_{\epsilon_{m,i}}(z_{m,i})\}$.
Since $B_{2\epsilon_{m,i}}(z_{m,i}) \supset U_{m,i}$ for every $i \in \mathbb{N}$, the family is a covering of $\Gamma$. By the Vitali $5B$-covering lemma, there is a subfamily $\{\widetilde{B}_{k}\}:=\{B_{2\epsilon_{m,i_{k}}}(z_{m,i_{k}})\}_{k=1}^{\widetilde{N}_{m}}$, $\widetilde{N}_{m} \in \mathbb{N} \cup \{+\infty\}$, such that  the family $\{\widetilde{B}_{k}\}$ is disjoint and  $\Gamma \subset \cup_{k=1}^{\widetilde{N}_{m}}B_{10\epsilon_{m,i_{k}}}(z_{m,i_{k}})$.
Hence, by \eqref{eqq.1} we have
$$
\sum_{k=1}^{\widetilde{N}_{m}}\epsilon_{m,i_{k}} = \frac{1}{10}\sum_{k=1}^{\widetilde{N}_{m}}10\epsilon_{m,i_{k}} \geq \frac{1}{20}\mathcal{H}^{1}_{10\delta_{m}} > m.
$$
Let $N_{m} \in \mathbb{N}$ be the minimal $N \in \mathbb{N}$ for which $\sum_{k=1}^{N}\epsilon_{m,i_{k}} > m.$
In what follows, by $x_{m,k}$, $k \in \{1,...,N_{m}\}$, we denote the points $z_{m,i_{k}}$, $k \in \{1,...,N_{m}\}$, which are ordered in accordance with the orientation of $\gamma$.
More precisely, if $t_{m,k}:=\min\{t:t \in \gamma^{-1}(x_{m,k})\}$, $k \in \{1,...,N_{m}\}$, then $t_{m,k} < t_{m,k+1}$ for every $k \in \{1,...,N_{m}-1\}$.

\textit{Step 2.} Throughout this step, we fix $m \in \mathbb{N}$ and put
$$
\widetilde{h}(x_{m,k}):=(-1)^{k}\epsilon_{m,k}, \quad k \in \{1,...,N_{m}\}.
$$
By the construction, $B_{2\epsilon_{m,i}}(x_{m,i}) \cap B_{2\epsilon_{m,j}}(x_{m,j})$ for any different $i,j \in \{1,...,N_{m}\}$. This gives the crucial inequality
$$
|\widetilde{h}_{m}(x_{m,i})-\widetilde{h}_{m}(x_{m,j})| \le \epsilon_{m,i}+\epsilon_{m,j} \le \operatorname{d}(x_{m,i},x_{m,j}).
$$
As a result, we have $\widetilde{h}_{m} \in \operatorname{LIP}_{1}(E_{m})$ with $E_{m}:=\{x_{m,k}\}_{k=1}^{N_{m}}$. Using Lemma \ref{Lm.McShane_Whitney} we find a function
$h_{m} \in \operatorname{LIP}_{1}(\operatorname{X})$ such that $h_{m}|_{E_{m}}=\widetilde{h}_{m}$.
By the construction, we clearly have
\begin{equation}
\label{eqq.17''}
\operatorname{V}_{h_{m} \circ \gamma}([a,b]) \geq \sum\limits_{k=1}^{N_{m}-1}|h_{m}(x_{m,k})-h_{m}(x_{m,k+1})| \geq \sum\limits_{k=1}^{N_{m}-1}r_{m,k}+r_{m,k+1} \geq m.
\end{equation}

\textit{Step 3.} Since $\gamma \in C([a,b],\operatorname{X})$, the set $\Gamma$ is compact. In particular, $\Gamma$ is a bounded set.
According to our construction $\{h_{m}(x_{m,1})\} \subset \mathbb{R}$ is a bounded sequence.
Hence, taking into account Lemma \ref{Lm.1} we combine Lemma \ref{Lm.Banah_Steinhauz} with \eqref{eqq.17''}. This gives existence of a function $\underline{h} \in \operatorname{LIP}(\operatorname{X})$
such that $\underline{h} \circ \gamma \notin BV([a,b],\operatorname{X})$. But this contradicts the assumptions of the theorem.

The proof is complete.
\end{proof}

\section{Good Lipschitz functions on $1$-rectifiable sets}
In fact, the main result of this section can be seen as a weaker version of Theorem 7.6 from \cite{Bate}. However, we present
an alternative elementary proof. In contrast to \cite{Bate} this leads to explicit constructions of Lipschitz functions
satisfying some nice infinitesimal properties.

Throughout the section, we fix an arbitrary (nonempty) metric space $\operatorname{X}=(\operatorname{X},\operatorname{d})$ and
a curve $\gamma \in C([a,b],\operatorname{X})$ such that $\mathcal{H}^{1}(\Gamma) \in (0,+\infty)$ (where $\Gamma = \gamma([a,b])$, as usual).
Keeping in mind Proposition \ref{Prop.rectifiability}, we fix a decomposition $\operatorname{\Gamma} = E \bigcup \cup_{i \in I}\Gamma^{i}$ such that $\mathcal{H}^{1}(E)=0$
and $\{\Gamma^{i}\}_{i \in I}$ is an at most countable sequence of images of curves with the corresponding properties. In this section,
we will assume that $I$ is countable and identify it with $\mathbb{N}$. The case of finite $I$ is ideologically similar but simpler in technical details.

We fix an arbitrary sequence $\{\rho_{m}\} \subset (0,+\infty)$ such that $\rho_{m} \downarrow 0$, $m \to \infty$, and
\begin{equation}
\notag
\rho_{m} < \frac{1}{2}\min\{\operatorname{diam}\Gamma^{i}:1 \le i \le m \hbox{ and } \Gamma^{i} \neq \emptyset\}, \quad m \in \mathbb{N}.
\end{equation}
We also set
\begin{equation}
\label{eqq.set_gamma_m}
\Gamma_{m}:=\Bigl(\bigcup_{i=1}^{m} \Gamma^{i}\Bigr) \setminus \Bigl(\bigcup_{i=1}^{m-1} B_{\rho_{m}}(x_{i})\Bigr), \quad m \in \mathbb{N}.
\end{equation}
Finally, we fix an arbitrary sequence $\{\epsilon_{m}\} \subset (0,+\infty)$ such that $\epsilon_{m} \downarrow 0$, $m \to \infty$, and
\begin{equation}
\label{eqq.41'}
\epsilon_{m} < \min\limits_{\substack{1 \le i,j \le m \\ i \neq j}}\inf\{\operatorname{d}(x',x''):x' \in \Gamma_{m} \cap \Gamma^{i}, x'' \in \Gamma_{m} \cap \Gamma^{j}\}=:r_{m}, \quad m \in \mathbb{N}.
\end{equation}

\subsection{Keystone family of sawtooth functions}

Given $i \in \mathbb{N}$, let $\gamma^{i}_{s}$ be the arc length parametrization of the curve $\gamma^{i}$.
Given $m \in \mathbb{N}$, if $\epsilon_{m} \geq \operatorname{diam}\Gamma^{i}$, we put
\begin{equation}
\label{eqq.3.1'}
g^{i}_{m}(x):=(\gamma^{i}_{s})^{-1}(x), \quad x \in \Gamma^{i}.
\end{equation}
If $\epsilon_{m} < \operatorname{diam}\Gamma^{i}$, we put
\begin{equation}
\label{eqq.3.2'}
g^{i}_{m}(x):=(\widetilde{g}^{i}_{m} \circ (\gamma^{i}_{s})^{-1})(x), \quad x \in \Gamma^{i}.
\end{equation}
In this formula, for each $l \in \{0,...,[l(\gamma^{i})/\epsilon_{m}]\}$
\begin{equation}
\label{eqq.3.3'}
\widetilde{g}^{i}_{m}(t):=
\begin{cases}
t-l\epsilon_{m} \quad \text{if $l$ is an odd number and $t \in [l \epsilon_{m},\min\{(l+1)\epsilon_{m},l(\gamma^{i})\}]$};\\
l\epsilon_{m}-t \quad \text{if $l$ is an even number and $t \in [l\epsilon_{m},\min\{(l+1)\epsilon_{m},l(\gamma^{i})\}]$}.
\end{cases}
\end{equation}
In the forthcoming subsection we show that the family $\{g_{i,m}\}$ satisfies some nice properties.

\subsection{Infinitesimal properties of sawtooth functions}
The first crucial observation is an immediate consequence of Lemma \ref{Lm.inner_and_outer_metric}.
\begin{Lm}
\label{Lm.derivative_keystone_sequence}
For each $i \in \mathbb{N}$ and every $m \in \mathbb{N}$,
\begin{equation}
\label{eqq.4.18'}
\lim\limits_{\substack{y \to x \\ y \in \Gamma^{i}}}\frac{|g_{m}^{i}(x)-g_{m}^{i}(y)|}{\operatorname{d}(x,y)} = 1  \quad \hbox{for each} \quad x \in \underline{\Gamma}^{i} \setminus
\Bigl\{\gamma^{i}_{s}(l\epsilon_{m}):l=0,...,\Bigl[\frac{l(\gamma^{i})}{\epsilon_{m}}\Bigr]\Bigr\}.
\end{equation}
\end{Lm}

The following lemma will be crucial for us.
\begin{Lm}
\label{Lm.keystone_function}
For each $i \in \mathbb{N}$, for every $\delta > 0$
there exists a sequence $\{\Gamma^{i}_{m}(\delta)\}_{m=1}^{\infty}$ of compact subsets of $\Gamma^{i}$ such that
the following properties hold:

(1) $\Gamma^{i}_{m+1}(\delta) \subset \Gamma^{i}_{m}(\delta)$ for every $m \in \mathbb{N}$;

(2) $\mathcal{H}^{1}(\Gamma^{i} \setminus \Gamma^{i}_{m}(\delta)) \to 0$, $m \to \infty$;

(3)  $g^{i}_{m} \in \operatorname{LIP}_{1+\delta}(\Gamma^{i}_{m}(\delta))$ for all $m \in \mathbb{N}$.

\end{Lm}

\begin{proof}
We fix $i \in \mathbb{N}$ and an arbitrary point $\underline{x} \in \underline{\Gamma}^{i}$.
Using \eqref{eqq.21}, we take $\sigma^{i}_{\underline{x}} > 0$ such that
\begin{equation}
\notag
l(\gamma^{i}_{\underline{x},x}) \le (1+\delta)\operatorname{d}(\underline{x},x) \quad \hbox{for all $x \in B_{\sigma^{i}_{\underline{x}}}(\underline{x}) \cap \Gamma^{i}$}.
\end{equation}
From \eqref{eqq.4.20} it is easy to see that the function
$\Gamma^{i}\setminus\{\underline{x}\} \ni x \to \psi^{i}_{\underline{x}}(x):=l(\gamma^{i}_{\underline{x},x})/\operatorname{d}(\underline{x},x)$ is continuous
on the set $\Gamma^{i}\setminus\{\underline{x}\}$. Since $\Gamma^{i} \setminus B_{\sigma^{i}_{\underline{x}}}(\underline{x})$ is a compact set, the function
$\psi^{i}_{\underline{x}}$ attains its maximum $\operatorname{M}^{i}_{\underline{x}} > 0$ on this set.

We put $\underline{\Gamma}^{i}_{\rm adm}:=\underline{\Gamma}^{i} \setminus \bigcup_{m \in \mathbb{N}}
\{\gamma^{i}_{s}(l\epsilon):l=0,...,[l(\gamma^{i})/\epsilon_{m})]\}.$
We claim that
\begin{equation}
\label{eqq.4.317}
\varlimsup\limits_{m \to \infty}\operatorname{Lip}g^{i}_{m}(\underline{x}) \le (1+\delta) \quad \hbox{for each} \quad \underline{x} \in \underline{\Gamma}^{i}_{\rm adm}.
\end{equation}
Indeed, we fix $\underline{x} \in \underline{\Gamma}^{i}_{\rm adm}$ and note that since $|g^{i}_{m}(\underline{x})-g^{i}_{m}(x)| \le l(\gamma^{i}_{\underline{x},x})$ for all $x \in \Gamma^{i}$ we have
\begin{equation}
\label{eqq.4.318}
\operatorname{Lip}g^{i}_{m}(\underline{x}) \le \max\{1+\delta,\operatorname{M}^{i}_{\underline{x}}\} \quad \hbox{for all $m \in \mathbb{N}$}.
\end{equation}
Hence, there are two cases to be considered. In the first case, $\operatorname{M}^{i}_{\underline{x}} \le (1+\delta)$ and \eqref{eqq.4.317} immediately follows from
\eqref{eqq.4.318}. In the second case, $\operatorname{M}^{i}_{\underline{x}} > (1+\delta)$ and by \eqref{eqq.4.18'} we get
\begin{equation}
\label{eqq.4.319}
\operatorname{m}^{i}_{\underline{x}}:=\inf\Bigl\{\operatorname{d}(\underline{x},x): x \in \Gamma^{i} \hbox{ and } \frac{l(\gamma^{i}_{\underline{x},x})}{\operatorname{d}(\underline{x},x) } > 1+\delta\Bigr\} > 0.
\end{equation}
As a result, since $l(\gamma^{i}_{\underline{x},x}) \geq \operatorname{d}(\underline{x},x)$ and $|g^{i}_{m}(x)-g^{i}_{m}(\underline{x})| \le 2\epsilon_{m}$, for each $m \in \mathbb{N}$ and
for every $x \in \Gamma^{i}$, we have the following estimate
\begin{equation}
\label{eqq.4.320}
\sup\limits_{\operatorname{d}(\underline{x},x) \geq \operatorname{m}^{i}_{\underline{x}}} \frac{|g^{i}_{m}(\underline{x})-g^{i}_{m}(x)|}{\operatorname{d}(\underline{x},x)} =
\sup\limits_{\operatorname{d}(\underline{x},x) \geq \operatorname{m}^{i}_{\underline{x}}} \frac{|g^{i}_{m}(\underline{x})-g^{i}_{m}(x)|}{l(\gamma_{\underline{x},x})} \frac{l(\gamma^{i}_{\underline{x},x})}{\operatorname{d}(\underline{x},x)} \le \frac{2\epsilon_{m}}{\operatorname{m}^{i}_{\underline{x}}}\operatorname{M}^{i}_{\underline{x}}.
\end{equation}
Passing to the limit in the right hand side of \eqref{eqq.4.320} as $m \to \infty$, keeping in mind that $\operatorname{M}^{i}_{\underline{x}} < +\infty$ and taking into account
\eqref{eqq.4.319}, we deduce \eqref{eqq.4.317}.

\textit{Step 5.} Given $m \in \mathbb{N}$, we put
$$
\Gamma^{i}_{m}(\delta):=\{x \in \Gamma^{i}:\operatorname{Lip}g^{i}_{k}(x) \le (1+\delta) \hbox{ for all } k \geq m\}.
$$
It is easy to see that $\Gamma^{i}_{m}(\delta)$ is closed for every $m \in \mathbb{N}$. Since $\Gamma^{i}$ is the image of a Lipschitz map $\gamma^{i}$, the sets
$\Gamma^{i}$ and thus $\Gamma^{i}_{m}(\delta)$ are in fact compact.
Furthermore, the sequence $\{\Gamma^{i}_{m}(\delta)\}_{m=1}^{\infty}$ satisfies properties (1) and (3). Finally, using Egorov's theorem and \eqref{eqq.4.317}
we deduce that $\mathcal{H}^{1}(\Gamma^{i} \setminus \Gamma^{i}_{m}(\delta)) \to 0$, $m \to \infty$.

The lemma is proved.
\end{proof}

Now we are ready to establish the \textit{main result of this section.}

\begin{Th}
\label{Th.key_function}
For each $\delta > 0$ there exist a sequence of compact subsets $\{\Gamma_{m}(\delta)\}$ of $\Gamma$, a sequence of functions
$\{h_{m}\} \subset \operatorname{LIP}_{1+\delta}(\operatorname{X})$,
and a sequence $\{\operatorname{N}_{m}\} \subset \mathbb{N}$ such that:

(1) $\Gamma_{m}(\delta) \subset \bigcup\limits_{i=1}^{m}\Gamma^{i}$ and $\Gamma_{m}(\delta) \subset \Gamma_{m+1}(\delta)$  for every $m \in \mathbb{N}$;

(2) $\mathcal{H}^{1}(\Gamma \setminus \Gamma_{m}(\delta)) \to 0$, $m \to \infty$;

(3) $\#\Bigl(h^{-1}_{m}(y) \cap \Gamma_{m}(\delta)\Bigr) \le \operatorname{N}_{m}$ for every $y \in \mathbb{R}$;

(4) for every $x \in \Gamma_{m}(\delta)$
\begin{equation}
\label{eqq.3.9'}
\lim\limits_{\substack{y \to x \\ y \in \Gamma}}\frac{|h_{m}(y)-h_{m}(x)|}{\operatorname{d}(x,y)}  = 1.
\end{equation}
\end{Th}

\begin{proof}
We fix an arbitrary $\delta > 0$.
Using Proposition \ref{Prop.rectifiability}, given $i \in \mathbb{N}$ with $i \geq 2$, we put $x_{i}:=\Gamma^{i} \bigcap \bigcup_{j=1}^{i-1}\Gamma^{j}$.
Now we recall \eqref{eqq.set_gamma_m} and, for each $m \in \mathbb{N}$, we put
\begin{equation}
\label{eqq.3.10'}
\Gamma_{m}(\delta):=\Gamma_{m} \bigcap \bigcup_{i=1}^{m} \Gamma^{i}_{m}(\delta),
\end{equation}
where  $\Gamma^{i}_{m}(\delta)$ is the same as in Lemma \ref{Lm.keystone_function}. Since $\Gamma_{m}$ and $\Gamma^{i}_{m}(\delta)$, $i \in \{1,...,m\}$ are compact sets,
the set $\Gamma_{m}(\delta)$ is compact.

For each $m \in \mathbb{N}$, we set
\begin{equation}
\label{eqq.3.11'}
\widetilde{h}_{m}:=\sum\limits_{i=1}^{m}\chi_{\Gamma_{m}(\delta)} g^{i}_{m},
\end{equation}
where for each $i \in \mathbb{N}$, the functions $g^{i}_{m}$, $m \in \mathbb{N}$, are the same as in  Lemma \ref{Lm.keystone_function}.

At the beginning of this section we assumed that $\mathcal{H}^{1}(\Gamma) \in (0,+\infty)$. Hence, taking into account that $\Gamma = \cup_{i=1}^{\infty}\Gamma^{i}$ and
$\mathcal{H}^{1}(\Gamma^{j} \cap \cup_{i=1}^{j-1}\Gamma^{i}) = 0$ for all $j \geq 2$, we have
$\mathcal{H}^{1}(\cup_{i=m}^{\infty}\Gamma^{i}) \to 0$, $m \to \infty$. Since the sets $\Gamma^{i}$, $i \in \mathbb{N}$, are connected, this implies that
$\operatorname{diam}\Gamma^{i} \to 0$, $i \to \infty$. The above facts together with \eqref{eqq.set_gamma_m} and \eqref{eqq.3.10'} implies that properties (1) and (2)
of the sets $\Gamma_{m}(\delta)$, $m \in \mathbb{N}$ hold. By \eqref{eqq.3.2'}, \eqref{eqq.3.3'} and \eqref{eqq.3.11'} we conclude that
property (3) holds with
$$
\operatorname{N}_{m}:=\sum\limits_{i=1}^{m}\Bigl(1+\Bigl[\frac{l(\gamma^{i})}{\epsilon_{m}}\Bigr]\Bigr).
$$

By \eqref{eqq.3.11'} and Lemma \ref{Lm.keystone_function} we have $\widetilde{h}_{m} \in \operatorname{LIP}_{1+\delta}(\Gamma^{i}_{m}(\delta))$ for each $i \in \{1,...,m\}$.
At the same time, by \eqref{eqq.3.2'} and \eqref{eqq.3.3'} we have $\sup_{x \in \operatorname{X}}|\widetilde{h}_{m}(x)| \le \epsilon_{m}$.
Combining these observations with \eqref{eqq.41'}, \eqref{eqq.3.10'} and \eqref{eqq.4.18'} it is easy to see that $\widetilde{h}_{m} \in \operatorname{LIP}_{1+\delta}(\Gamma_{m}(\delta))$ and
\begin{equation}
\label{eqq.4.314}
\lim\limits_{\substack{y \to x \\ y \in \Gamma_{m}(\delta)}}\frac{|\widetilde{h}_{m}(y)-\widetilde{h}(x)|}{\operatorname{d}(x,y)}  = 1 \quad \hbox{for all} \quad x \in \Gamma_{m}(\delta).
\end{equation}
Now using Lemma \ref{Lm.McShane_Whitney} we deduce that there exists $h_{m} \in \operatorname{LIP}_{1+\delta}(\operatorname{X})$ such that $h_{m}|_{\Gamma_{m}(\delta)}=\widetilde{h}_{m}$.
As a result, combining \eqref{eqq.4.314} with Lemma \ref{Lm.local_lip_const} we conclude.
\end{proof}

\section{Tracking $AC_{p}$-regularity via post-compositions with Lipschitz functions}

Throughout this section, we fix a metric space $\operatorname{X}=(\operatorname{X},\operatorname{d})$ and real numbers $a < b$.
We divide this section into two subsections. In the first one, we present a proof of Theorem \ref{Th.main_1} based on
a powerful but not fully constructive result from \cite{Bate}.
In the second one, we present our alternative elementary approach to the proof of Theorem \ref{Th.main_1}
based on the construction given in \S 4.

\subsection{Indirect proof of Theorem \ref{Th.main_1}}
In this subsection, by $\operatorname{BLIP}_{1}(\operatorname{X})$ we denote the set of all bounded $1$-Lipschitz functions on $\operatorname{X}$.
Recall that a set $S \subset \operatorname{BLIP}_{1}(\operatorname{X})$ is  \textit{residual} (with respect to the $\sup$-norm) if it contains a countable
intersection of open dense subsets.

Combining  Theorem 7.6 from \cite{Bate} with Proposition \ref{Prop.rectifiability}, we immediately get the following crucial result.
\begin{Prop}
\label{Prop.Bate}
Let $\gamma \in C([a,b],\operatorname{X})$ be such that $\mathcal{H}^{1}(\Gamma) < +\infty$. Then the set
$$
S:=\{h \in \operatorname{BLIP}_{1}(\operatorname{X}): \mathcal{J}_{\Gamma}h(x) = 1 \hbox{ for $\mathcal{H}^{1}$-a.e. $x \in \Gamma$}\}
$$
is residual in $\operatorname{BLIP}_{1}(\operatorname{X})$.
\end{Prop}

\textit{Proof of Theorem \ref{Th.main_Bate}.} By Lemma \ref{Lm.continuity} we deduce that $\gamma \in C([a,b],\operatorname{X}$).
By Theorem \ref{Th.1} we have $\mathcal{H}^{1}(\Gamma) < +\infty$. An application of Proposition \ref{Prop.Bate} gives the existence
of a function $h \in \operatorname{BLIP}_{1}(\operatorname{X})$ such that $\mathcal{J}_{\Gamma}h(x) = 1$  for $\mathcal{H}^{1}$-a.e. $x \in \Gamma$.
By Propositions \ref{Prop.area_formula} and \ref{Prop.integral_repres_for_variation},
\begin{equation}
\begin{split}
&V_{h \circ \gamma}([a,b])=\int\limits_{h(\Gamma)}\#(h \circ\gamma)^{-1}(y)\,dy = \int\limits_{h(\Gamma)}\sum\limits_{x \in h^{-1}(y)}\#\gamma^{-1}(x)\,dy  \\
&=\int\limits_{\Gamma}
\#\gamma^{-1}(x)\mathcal{J}_{\Gamma}h(x)\,d\mathcal{H}^{1}(x) = \int\limits_{\Gamma}
\#\gamma^{-1}(x)\,d\mathcal{H}^{1}(x)=V_{\gamma}([a,b]).
\end{split}
\end{equation}
Hence $\gamma \in BV([a,b],\operatorname{X})$. Finally, replacing in the above formula $[a,b]$ by an arbitrary closed interval $[c,d]$ and making minor modifications, we get \eqref{eqq.variation}
and complete the proof.
\hfill$\Box$

Now we are ready to prove the first main result of this paper.

\textit{Proof of Theorem \ref{Th.main_1}.} Application of Theorem \ref{Th.main_Bate} implies that $\gamma \in C([a,b],\operatorname{X}) \cap BV([a,b],\operatorname{X})$.
Now we show that $\gamma \in AC_{1}([a,b],\operatorname{X})$. If, for some open interval $(c,d) \subset [a,b]$,
we put
$$
V_{\gamma}((c,d)):=V_{\gamma}([c,d])=\sup\limits_{[c',d']\subset (c,d)}V_{\gamma}([c',d']),
$$
and for any open (in the usual topology of the real line) set $G \subset [a,b]$,
which always can be represented as a union of at most countable family $\{(c_{i},d_{i})\}_{i \in I}$ of disjoint intervals, we put
$$
V_{\gamma}(G):=\sum\limits_{i \in I}V_{\gamma}((c_{i},d_{i})),
$$
then by Theorem 4.4.8 from \cite{HKST} we obtain a Radon  measure $\nu_{\gamma}$ on $[a,b]$.

By Theorem \ref{Th.main_Bate}, it follows easily that for some $h \in \operatorname{LIP}(\operatorname{X})$ we have $V_{\gamma}(G)=V_{h\circ\gamma}(G)$
for every open (in the usual topology of the real line) set $G \subset [a,b]$. The corresponding Radon measure $\nu_{h\circ\gamma}$ coincide with $\nu_{\gamma}$ on any Borel set $E \subset [a,b]$.
On the other hand, $h \circ \gamma$ belongs to the space $AC_{1}([a,b])$. This clearly implies that $\nu_{h\circ\gamma}$ and $\nu_{\gamma}$ are absolutely
continuous with respect to the measure $\mathcal{L}^{1}$ restricted to $[a,b]$. Hence $\gamma \in AC_{1}([a,b],\operatorname{X})$.

Finally, in order to show that $\gamma \in AC_{p}([a,b],\operatorname{X})$ it is sufficient to note that by the $N$-Luzin property of $\gamma$ and Lemma \ref{Lm.Jacobian_Lip_constant}
we have $|\dot{\gamma}(t)|=|(h\circ\gamma)'(t)|$ for $\mathcal{L}^{1}$-a.e. $t \in [a,b]$.
\hfill$\Box$

\subsection{Direct proof of Theorem \ref{Th.main_1}}
We start with the elementary fact about absolutely continuous functions. We present the details for the completeness.
\begin{Lm}
\label{elementary_lemma}
Let $\{\alpha_{j}\} \subset (0,+\infty)$ be such that $\sum_{j =1}^{\infty}\alpha_{j} < +\infty$. Let $\{f_{j}\} \subset AC_{1}([a,b])$ be such that
\begin{equation}
\notag
\sum\limits_{j=1}^{\infty}\alpha_{j}\|f'_{j}\|_{L_{1}([a,b])} < +\infty.
\end{equation}
Then the series $\sum_{j=1}^{\infty}\alpha_{j}f_{j}$ converges uniformly to some function $f \in AC_{1}([a,b])$ and
\begin{equation}
\label{eqq.41}
|f'(t)| \le \sum\limits_{j=1}^{\infty}\alpha_{j}|f'_{j}(t)| \quad \hbox{for $\mathcal{L}^{1}$-a.e.  $t \in [a,b]$}.
\end{equation}
\end{Lm}
\begin{proof}
By the assumptions of the lemma, the series $\sum_{j=1}^{\infty}\alpha_{j}f_{j}$ converges uniformly to some $f \in C([a,b])$.
Hence, given $a \le t_{1} < t_{2} \le b$ by the Beppo Levi Theorem we have
$$
|f(t_{1})-f(t_{2})| \le \sum\limits_{j=1}^{\infty}\alpha_{j}|f_{j}(t_{1})-f_{j}(t_{2})| \le \sum\limits_{j=1}^{\infty}\alpha_{j}\int\limits_{t_{1}}^{t_{2}}|f'_{j}(t)|\,dt =  \int\limits_{t_{1}}^{t_{2}}\sum\limits_{j=1}^{\infty}\alpha_{j}|f'_{j}(t)|\,dt.
$$
By Definition \ref{Def.4} (for the particular case of the real-valued curve), this implies that $f \in AC_{1}([a,b])$ and, furthermore, \eqref{eqq.41} holds.
\end{proof}

The following lemma is quite important for the sequel. It has the same flavor as Lemma \ref{Lm.Banah_Steinhauz} and can
be proved by a similar method.
\begin{Lm}
\label{Lm.Bate_substitution}
Let $\gamma \in AC_{1}([a,b],\operatorname{X})$ and $p \in (1,\infty]$. Assume that there exists a family of Borel sets $\{\Gamma_{m}\}_{m \in \mathbb{N}}$ and functions $\{h_{m}\} \subset \operatorname{LIP}^{\rm str}(\Gamma)$ such that
the following holds:

(1) $\Gamma_{m} \subset \Gamma_{m+1} \subset \Gamma$ for all $m \in \mathbb{N}$ and $\mathcal{H}^{1}(\Gamma \setminus \Gamma_{m}) \to 0$, $m \to \infty$;

(2) $\lim_{m \to \infty}\||\dot{\gamma}|\|_{L_{p}}(\gamma^{-1}(\Gamma_{m}))  = +\infty$;

(3) $(h_{m} \circ \gamma)' \in L_{p}([a,b])$ for all $m \in \mathbb{N}$;

(4) $|(h_{m} \circ \gamma)'(t)|=|\dot{\gamma}(t)|$ for $\mathcal{L}^{1}$-a.e. $t \in \gamma^{-1}(\Gamma_{m})$.

Then there exists a function $\underline{h} \in \operatorname{LIP}^{\rm str}(\Gamma) \cap \operatorname{LIP}(\operatorname{X})$ such that $\underline{h} \circ \gamma \notin AC_{p}([a,b])$.

\end{Lm}

\begin{proof} For each $m \in \mathbb{N}$, we define the countably subadditive positively homogeneous functional on  the space $\operatorname{LIP}^{\rm str}(\Gamma)$ be letting
$$
p_{m}(h):=\|(h\circ\gamma)'\|_{\gamma^{-1}(\Gamma_{m})}, \quad h \in \operatorname{LIP}^{\rm str}(\Gamma).
$$
It follows from (1)--(4) that the sequence $\{p_{m}\}_{m=1}^{\infty}$ satisfies all the assumptions of Proposition \ref{Prop.Banach_Steinhaus_substitution}.
Hence taking into account Remark \ref{Rem.Lip_Banach} we deduce existence of $\underline{h} \in \operatorname{LIP}^{\rm str}(\Gamma)$ such that
$$
\|(h \circ \gamma)'\|_{L_{p}([a,b])} = \sup\limits_{m \in \mathbb{N}}p_{m}(h) = +\infty.
$$
By Lemma \ref{Lm.McShane_Whitney}, we get $\underline{h} \in \operatorname{LIP}(\operatorname{X}) \cap \operatorname{LIP}^{\rm str}(\Gamma)$, which completes the proof.
\end{proof}

\begin{Th}
\label{Th.total_variation}
Let a curve $\gamma \in C([a,b],\operatorname{X})$ be such that $h \circ \gamma \in BV([a,b])$ for each function
$h \in \operatorname{LIP}(\operatorname{X})$. Then
$\gamma \in BV([a,b])$.
\end{Th}

\begin{proof}
Assume on the contrary that $V_{\gamma}([a,b]) = +\infty$ and fix $\delta > 0$. We apply Theorem \ref{Th.key_function} and find
a sequence $\{\Gamma_{m}(\delta)\}_{m = 1}^{\infty}$ of subsets of $\Gamma$ and a sequence of functions
$\{h_{m}\}_{m=1}^{\infty} \subset \operatorname{LIP}_{1+\delta}(\operatorname{X})$
with the corresponding properties.
By Proposition \ref{Prop.integral_repres_for_variation}, for each $j \in \mathbb{N}$ there is $\Gamma_{n_{j}}(\delta)$ such that
$\int\limits_{\Gamma_{n_{j}}(\delta)}\#\gamma^{-1}(x)\,d\mathcal{H}^{1}(x) \geq j.$
Since the set $\Gamma_{n_{j}}(\delta)$ is 1-rectifiable, we obtain
\begin{equation}
\begin{split}
&V_{h_{n_{j}}}([a,b])=\int\limits_{h_{n_{j}}(\Gamma)}\#(h_{n_{j}}\circ\gamma)^{-1}(y)\,dy = \int\limits_{h_{n_{j}}(\Gamma)}\sum\limits_{x \in h^{-1}_{n_{j}}(y)}\#\gamma^{-1}(x)\,dy  \\
&=\int\limits_{\Gamma}
\#\gamma^{-1}(x)\mathcal{J}_{\Gamma}h_{n_{j}}(x)\,d\mathcal{H}^{1}(x) \geq \int\limits_{\Gamma_{n_{j}}}
\#\gamma^{-1}(x)\,d\mathcal{H}^{1}(x).
\end{split}
\end{equation}
As a result, for each $j \in \mathbb{N}$ we deduce the existence of a function $h_{n_{j}} \in \operatorname{LIP}_{1+\delta}(\operatorname{X})$ such that $V_{h_{n_{j}}}([a,b]) \geq j$.
Combination of this observation with Lemma \ref{Lm.Banah_Steinhauz} completes the proof.
\end{proof}

\begin{Th}
\label{Th.ac_property}
Let $\gamma \in C([a,b],\operatorname{X})$ be such that $h \circ \gamma$ satisfies the $N$-Luzin property
for every function $h \in \operatorname{LIP}(\operatorname{X})$. Then
$\gamma$ satisfies the $N$-Luzin property.
\end{Th}

\begin{proof}
Assume on the contrary that $\gamma$ fails to satisfy the $N$-Luzin property. Then, there exists a set $E \subset [a,b]$ with $\mathcal{L}^{1}(E) = 0$
such that $\mathcal{H}^{1}(\gamma(E)) > 0$. We apply Theorem \ref{Th.key_function} and fix an arbitrary $m \in \mathbb{N}$ such that $\mathcal{H}^{1}(\gamma(E) \cap \Gamma_{m}(1)) > 0$. Hence, by Proposition \ref{Prop.area_formula}
\begin{equation}
\notag
\begin{split}
&\mathcal{L}^{1}\Bigl((h_{m} \circ \gamma)(E)\Bigr) \geq \frac{1}{\operatorname{N}_{m}}\int\limits_{(h_{m} \circ \gamma)(E)}\#(h_{m}^{-1}(t)\cap \Gamma_{m}(1))\,dt\\
&\geq \frac{1}{\operatorname{N}_{m}}\int\limits_{\gamma(E) \cap \Gamma_{m}(1)}\mathcal{J}_{\Gamma}h_{m}(x)\,d\mathcal{H}^{1}(x) \geq \frac{\mathcal{H}^{1}(\gamma(E) \cap \Gamma_{m}(1))}{\operatorname{N}_{m}} > 0.
\end{split}
\end{equation}
Hence, $h_{m} \circ \gamma$ fails to satisfy the $N$-Luzin property. This contradiction completes the proof.
\end{proof}

Now we are ready to prove the first main result of this paper.

\textit{Proof of Theorem \ref{Th.main_1}}. By Proposition \ref{Prop.BZ} and Theorems \ref{Th.total_variation}, \ref{Th.ac_property} we deduce that the curve $\gamma$ belongs to $AC_{1}([a,b],\operatorname{X})$.
Now an application of Lemma \ref{Lm.Bate_substitution} in combination with Theorem \ref{Th.key_function} proves the claim.
\hfill$\Box$

\section{Tracking $W^{1}_{p}$-regularity via post-compositions with Lipschitz functions}

Throughout the section, we fix a \textit{complete separable} metric space $\operatorname{X}=(\operatorname{X},\operatorname{d})$ and real numbers $a < b$.
For each $\epsilon > 0$, we fix a maximal $\epsilon$-separated subset
$\mathcal{N}_{\epsilon}:=\{x_{\epsilon,i}\}_{i \in I_{\epsilon}}$ of $\operatorname{X}$. Clearly, $\mathcal{N}_{\epsilon}$ is at most countable for every $\epsilon > 0$.
Given $x \in \operatorname{X}$, we put $h_{x}:=\operatorname{d}(x,\cdot)$.
If $\gamma \in \mathfrak{B}([a,b],\operatorname{X})$ is such that $[h_{x} \circ \gamma] \cap C([a,b],\operatorname{X}) \neq \emptyset$, then
$g^{\gamma}_{x}$ denotes a unique continuous representative of $[h_{x} \circ \gamma]$ and $E_{x}:=\{t \in [a,b]:g^{\gamma}_{x}(t)=h_{x}(\gamma(t))\}$.

Given a metric space $\operatorname{Y}=(\operatorname{Y},\rho)$ and a map $f \in \mathfrak{B}([a,b],\operatorname{Y})$, for each $\epsilon > 0$ and $\delta > 0$, we consider the sets
\begin{equation}
\begin{split}
\label{eqq.discontinuity_set}
&\mathcal{D}_{f}(\epsilon,\delta):=\{(t',t'') \in [a,b]^{2}:|t'-t''| < \delta \hbox{ and } \rho(f(t'),f(t'')) \geq \epsilon\};\\
&\mathcal{C}_{f}(\epsilon,\delta):=\{(t',t'') \in [a,b]^{2}:|t'-t''| < \delta \hbox{ and } \rho(f(t'),f(t'')) < \epsilon\}.
\end{split}
\end{equation}
If $[f] \cap C([a,b],\operatorname{Y}) \neq \emptyset$, then by $\overline{f}$ we denote a unique continuous representative of $\overline{f}$.

\begin{Remark}
Since $f$ is Borel, the sets $\mathcal{D}_{f}(\epsilon,\delta)$ and $\mathcal{C}_{f}(\epsilon,\delta)$ are Borel for all $\epsilon,\delta >0$, provided that $\operatorname{Y}=(\operatorname{Y},\rho)$ is separable.
\end{Remark}

The following assertion looks standard. However, we are not able to give a precise reference. Since the proof is rather subtle,
we present the details.
\begin{Prop}
\label{Prop.4.2}
Let $\operatorname{Y}=(\operatorname{Y},\rho)$ be a complete separable metric space. Given $f \in \mathfrak{B}([a,b],\operatorname{Y})$, $[f] \cap C([a,b],\operatorname{Y})=\emptyset$ if and only if there is $\underline{\epsilon} > 0$ such that $\mathcal{L}^{2}(\mathcal{D}_{f}(\underline{\epsilon},\delta)) > 0$ for every $\delta > 0$.
\end{Prop}

\begin{proof}
It is equivalent to show that $[f] \cap C([a,b],\operatorname{Y}) \neq \emptyset$ if and only if, for every $\epsilon >0$, there is $\delta(\epsilon) > 0$ such that
$\mathcal{L}^{2}(\mathcal{D}_{f}(\epsilon,\delta(\epsilon))) = 0$. We split the proof into several elementary steps.

\textit{Step 1.} If $[f] \cap C([a,b],\operatorname{Y}) \neq \emptyset$, then
there is a set $E \subset [a,b]$ with $\mathcal{L}^{1}(E)=0$ such that $f$ becomes continuous after redefining on $E$. We set $S=[a,b] \setminus E$ and note that $\mathcal{L}^{2}(E \times E) = 0$, $\mathcal{L}^{2}(E \times S) = 0$ and $\mathcal{L}^{2}(S \times E) = 0$
by the Fubini theorem. At the same time, since
$\overline{f}$ is uniformly continuous on $[a,b]$, for each $\epsilon > 0$ there is $\delta(\epsilon) > 0$
such that $\mathcal{C}_{f}(\epsilon,\delta)$ contains intersection of $S \times S$ with the set $\{(t',t'') \in [a,b]^{2}:|t'-t''| < \delta(\epsilon)\}$.
This implies that $\mathcal{L}^{2}(\mathcal{D}_{f}(\epsilon,\delta)) = 0$ and completes the verification of necessity.

\textit{Step 2.} For each $t \in [a,b]$ and $\delta > 0$ we set $U_{\delta}(t):=[a,b] \cap (t-\delta,t+\delta)$ for brevity.
To prove the sufficiency, for every $t \in [a,b]$ and every $\epsilon > 0$, we fix $\delta=\delta_{\epsilon}(t) > 0$ such that, for
$\mathcal{L}^{2}$-a.e. pair $(t',t'') \in U_{\delta}(t) \times U_{\delta}(t)$, we have $\rho(f(t'),f(t'')) < \epsilon$. We set $\widetilde{\delta}_{\epsilon}(t):=\min\{\epsilon,\delta_{\epsilon}(t)\}.$
By the Fubini theorem, for every $t \in [a,b]$ and $\epsilon > 0$, there is a set $\widetilde{S}_{\epsilon}(t) \subset U_{\widetilde{\delta}_{\epsilon}(t)}(t)$
with $\mathcal{L}^{1}(U_{\widetilde{\delta}_{\epsilon}(t)}(t) \setminus \widetilde{S}_{\epsilon}(t)) = 0$
such that $\mathcal{L}^{1}(U_{\widetilde{\delta}_{\epsilon}(t)}(t) \setminus \Pi_{\epsilon}(t'))=0$ for each $t' \in \widetilde{S}_{\epsilon}(t)$, where
\begin{equation}
\label{eqq.pi_set}
\Pi_{\epsilon}(t'):=\{t'' \in U_{\widetilde{\delta}_{\epsilon}(t)}(t):(t',t'') \in \mathcal{C}_{f}(\epsilon,\widetilde{\delta}_{\epsilon}(t))\}, \quad t' \in \widetilde{S}_{\epsilon}(t).
\end{equation}

\textit{Step 3.} We fix an arbitrary $t \in [a,b]$ and put $\delta^{i}=\widetilde{\delta}_{\frac{1}{i}}(t)$, $\widetilde{S}^{i}(t)=\widetilde{S}_{\frac{1}{i}}(t)$ for every $i \in \mathbb{N}$.
We set $S^{0}(t):=\widetilde{S}^{1}(t)$ and fix an arbitrary $t_{1}(t) \in \widetilde{S}^{1}(t) \setminus \{t\}$.
Arguing by induction, we can easily build a sequence $\{t_{n}(t)\} \subset [a,b] \setminus \{t\}$
and a sequence of sets $\{S^{n}(t)\}$ such that
for every $n \in \mathbb{N}$ the following properties hold
(we set $\Pi^{n}(t):=\Pi_{\frac{1}{n}}(t_{n}(t))$, $n \in \mathbb{N}$ for brevity):
\begin{itemize}

\item[\((S1)\)] $t_{n}(t)$ is an arbitrary point in $(S^{n-1}(t) \cap \widetilde{S}^{n}(t)) \setminus \{t\}$;

\item[\((S2)\)] $S^{n}(t):=\widetilde{S}^{n}(t) \cap \Pi^{n}(t) \cap S^{n-1}(t)$.

\end{itemize}
Indeed, the base of induction is the construction of $S^{0}(t),\widetilde{S}^{1}(t)$ and the above choice of $t_{1}(t)$.
The induction step is clear from $(S1)$ and $(S2)$.

The first crucial property which follows from $(S1)$ and $(S2)$ is that
\begin{equation}
\label{eqq.6.3''}
t_{m}(t) \in \Pi_{\frac{1}{n}}(t_{n}(t)) \quad \hbox{for each $n \in \mathbb{N}$ and every $m \in \mathbb{N}$ satisfying $m \geq n$}.
\end{equation}
The second crucial property is that
\begin{equation}
\label{eqq.6.4''}
S^{n}(t) \subset U_{\delta^{n}}(t) \quad \hbox{and} \quad \mathcal{L}^{1}(U_{\delta^{n}}(t) \setminus S^{n}(t)) = 0 \quad \hbox{for every $n \in \mathbb{N}$}.
\end{equation}

\textit{Step 4.} Given $t \in [a,b]$, by \eqref{eqq.discontinuity_set} -- \eqref{eqq.6.3''} we have
\begin{equation}
\label{eqq.Cauchy_sequence}
\lim\limits_{n \to \infty}t_{n}(t) = t \quad \hbox{and} \quad \rho(f(t_{l}(t)),f(t_{m}(t))) \le \frac{1}{m} \quad \hbox{for each $m,l \in \mathbb{N}$ with $l \geq m$}.
\end{equation}
By \eqref{eqq.Cauchy_sequence} and the completeness of $\operatorname{Y}$, for each $t \in [a,b]$, there is $A(t) \in \operatorname{Y}$ such that
\begin{equation}
\label{eqq.4.3}
\rho(f(t_{n}(t)), A(t)) \le \frac{1}{n} \quad \hbox{for every} \quad n \in \mathbb{N}.
\end{equation}
As a result, taking into account \eqref{eqq.discontinuity_set}, \eqref{eqq.pi_set}, $(S2)$ and \eqref{eqq.4.3}, for each  $n \in \mathbb{N}$, we get
\begin{equation}
\label{eqq.4.4}
\rho(f(t'),A(t)) \le \rho(f(t'),f(t_{n}(t)))+\rho(f(t_{n}(t)), A(t)) < \frac{2}{n} \quad \hbox{for every} \quad t' \in S^{n}(t).
\end{equation}

\textit{Step 5.}
Given $n \in \mathbb{N}$, the family $\{(t-\frac{\delta_{1/n}(t)}{2},t+\frac{\delta_{1/n}(t)}{2}):t \in [a,b]\}$ is a covering of the compact set $[a,b]$.
Hence, there is $\{t_{n,i}\}_{i=1}^{N_{n}} \subset [a,b]$ with $N_{n} \in \mathbb{N}$ such that
\begin{equation}
\label{eqq.4.5}
[a,b] \subset \bigcup\limits_{i=1}^{N_{n}}\Bigl(t_{n,i}-\frac{\delta_{n,i}}{2},t_{n,i}+\frac{\delta_{n,i}}{2}\Bigr),
\end{equation}
where
$\delta_{n,i}:=\delta_{1/n}(t_{n,i})$, $i \in \{1,...,N_{n}\}$.
By \eqref{eqq.6.4''}, we have $\mathcal{L}^{1}(D_{n}) = 0$ for each $n \in \mathbb{N}$, where
\begin{equation}
\label{eqq.4.6}
D_{n}: = \bigcup\limits_{i=1}^{N_{n}}(t_{n,i}-\delta_{n,i},t_{n,i}+\delta_{n,i}) \setminus S^{n}(t_{n,i}), \quad n \in \mathbb{N}.
\end{equation}
Given $n \in \mathbb{N}$, we set $\delta(n):=\min\{\frac{\delta_{n,i}}{2}, i=1,...,N_{n}\}$. Combining \eqref{eqq.4.4} -- \eqref{eqq.4.6}
we get
\begin{equation}
\label{eqq.4.7}
([a,b] \setminus D_{n}) \times ([a,b] \setminus D_{n})   \subset \mathcal{C}_{f}\Bigl(\frac{4}{n},\delta(n)\Bigr) \quad \hbox{for every} \quad n \in \mathbb{N}.
\end{equation}
Now, we put
$\underline{E}:=\cap_{n=1}^{\infty}[a,b] \setminus D_{n}$
and take into account \eqref{eqq.4.7}. This gives
$$
A(t)=\lim\limits_{\substack{t' \to t \\ t' \in \underline{E}}}f(t) \quad \hbox{for every} \quad t \in [a,b].
$$
Furthermore, $A(t)=f(t)$ for every $t \in \underline{E}$ and $\mathcal{L}^{1}([a,b] \setminus \underline{E})=0$. Hence, letting
$\widetilde{f}(t):=f(t)$ for $t \in \underline{E}$ and $\widetilde{f}(t):=A(t)$ for $t \in [a,b] \setminus \underline{E}$ we
obtain a map $\widetilde{f} \in C([a,b],\operatorname{Y})$ such that $\widetilde{f} \in [f]$. This completes the proof of the sufficiency.

The proposition is proved.
\end{proof}

The following assertion which is based on Proposition \ref{Prop.4.2} is crucial for our analysis.
\begin{Prop}
\label{Prop.4.3}
Let $\gamma \in \mathfrak{B}([a,b],\operatorname{X})$ be such that $[h\circ\gamma] \cap C([a,b]) \neq \emptyset$ for every $h \in \operatorname{LIP}(\operatorname{X})$.
Then, for each $\epsilon > 0$ and every $\underline{q} > 1$,
\begin{equation}
\label{eqq.62'}
[a,b] \subset \bigcup\limits_{i \in I_{\epsilon}}g^{-1}_{\epsilon,i}((-\underline{q}\epsilon,\underline{q}\epsilon)),
\end{equation}
where, for each $i \in I_{\epsilon}$, $g_{\epsilon,i}$ is a unique continuous representative of $h_{x_{\epsilon,i}} \circ \gamma$.

\end{Prop}

\begin{proof}
We fix $\epsilon > 0$, $q \in (1,\underline{q})$ and simplify some notation. We set $\mathcal{N}:=\mathcal{N}_{\epsilon}$, $I:=I_{\epsilon}$,
$x_{i}:=x_{\epsilon,i}$ and $g_{i}:=g_{\epsilon,i}$ for each $i \in I$. Furthermore, we put $\widetilde{G}_{i}:=g^{-1}_{i}((-\epsilon,\epsilon))$ and $G_{i}:=g^{-1}_{i}((-q\epsilon,q\epsilon))$
for every $i \in I$.
We set $E:=\cap_{i \in I}E_{x_{i}}$ and note that, by the maximality of $\mathcal{N}$, 
\begin{equation}
\label{eqq.61'}
E \subset \bigcup\limits_{i \in I}\widetilde{G}_{i}.
\end{equation}
Since $I$ is at most countable,
without loss of generality we may assume that $\mathcal{L}^{1}(\widetilde{G}_{i}) > 0$ for all $i \in I$ (otherwise, we modify $E$ by deleting all intersections $E \cap \widetilde{G}_{i}$
with $\mathcal{L}^{1}(\widetilde{G}_{i})=0$).

To prove \eqref{eqq.62'} we assume that $[a,b] \setminus E \neq \emptyset$ (otherwise, there is nothing to prove) and
fix $\underline{t} \in [a,b] \setminus E$. We put $U_{l}(\underline{t}):=(\underline{t}-1/l,\underline{t}+1/l) \cap [a,b]$, $l \in \mathbb{N}$.
It is sufficient to show that
\begin{equation}
\label{eqq.forgotten}
\mathcal{L}^{1}(U_{\underline{m}}(\underline{t}) \setminus G_{\underline{i}}) = 0 \quad \hbox{for some} \quad \hbox{$\underline{i} \in I$ and $\underline{m} \in \mathbb{N}$}.
\end{equation}
Indeed, by the continuity of $g_{i}$, $i \in I$, we have $\underline{t} \in g^{-1}_{\underline{i}}([-q\epsilon,q\epsilon]) \subset g^{-1}_{\underline{i}}((-\underline{q}\epsilon,\underline{q}\epsilon))$.
Since $\underline{t} \in [a,b] \setminus E$ was chosen arbitrarily, this gives \eqref{eqq.62'}.

To establish \eqref{eqq.forgotten} we assume the contrary, i.e.
\begin{equation}
\label{eqq.62''}
\mathcal{L}^{1}(U_{m}(\underline{t}) \setminus G_{i}) > 0 \quad \hbox{for each $i \in I$ and every $m \in \mathbb{N}$}.
\end{equation}
We fix an arbitrary $i_{1} \in I$ and suppose that, for some $j \in \mathbb{N}$, we have already chosen indices $i_{1},...,i_{j+1} \in I$ such that, for each $l \in \{1,...,j\}$,
\begin{equation}
\label{eqq.63'}
\mathcal{L}^{1}(U_{m}(\underline{t}) \setminus
\cup_{s=1}^{l} G_{i_{s}}) > 0 \quad \hbox{for every $m \in \mathbb{N}$}
\end{equation}
and, furthermore,
\begin{equation}
\label{eqq.64'}
\mathcal{L}^{1}\Bigl(U_{l}(\underline{t}) \setminus (\widetilde{G}_{i_{l+1}} \bigcup
\cup_{s=1}^{l} G_{i_{s}})\Bigr)
<
\mathcal{L}^{1}(U_{l}(\underline{t}) \setminus \cup_{s=1}^{l} G_{i_{s}}).
\end{equation}
By \eqref{eqq.61'} it is clear that \eqref{eqq.63'} and \eqref{eqq.64'} hold true for $j=1$.
If $\mathcal{L}^{1}(U_{m_{j}}(\underline{t}) \setminus \cup_{s=1}^{j+1} G_{i_{s}})=0$ for some $m_{j} \in \mathbb{N}$, then we stop. Otherwise, we continue the procedure.
As a result, there is either a finite set $\underline{I}=\{i_{j}\}_{j=1}^{\underline{j}} \subset I$ with $\underline{j} \in \mathbb{N}$ or
an infinite set $\underline{I}=\{i_{j}\}_{j=1}^{\infty} \subset I$
such that \eqref{eqq.63'} and \eqref{eqq.64'} hold either for every $l \in \{1,...,\underline{j}\}$ or for every $l \in \mathbb{N}$, respectively.

By \eqref{eqq.62''}--\eqref{eqq.64'}, for each $m \in \mathbb{N}$, there are
$i' \in \underline{I}$ and $i'' \in \underline{I}$ such that
\begin{equation}
\label{eqq.65'}
\mathcal{L}^{1}(U_{m}(\underline{t}) \cap G_{i'}) > 0 \quad \hbox{and} \quad  \mathcal{L}^{1}(U_{m}(\underline{t}) \setminus \widetilde{G}_{i''} \cup G_{i'}) < \mathcal{L}^{1}(U_{m}(\underline{t}) \setminus  G_{i'}).
\end{equation}

The crucial observation is that, given $j_{1} < j_{2}$,
\begin{equation}
\label{eqq.66'}
\operatorname{d}(\gamma(t_{1}),\gamma(t_{2})) \geq (q-1)\epsilon \quad \hbox{for every $t_{1} \in \widetilde{G}_{i_{j_{1}}} \cap E$ and $t_{2} \in (\widetilde{G}_{i_{j_{2}}} \setminus \cup_{j < j_{2}}G_{i_{j}}) \cap  E$}.
\end{equation}
Indeed, otherwise, by the triangle inequality $\operatorname{d}(\gamma(t_{2}),x_{i_{1}}) \le \operatorname{d}(\gamma(t_{2}),\gamma(t_{1}))+\operatorname{d}(x_{i_{1}},\gamma(t_{1})) < q\epsilon.$
Consequently, $t_{2} \in \widetilde{G}_{i_{j_{1}}}$. This contradicts the construction.

We set $S_{1}:=\widetilde{G}_{i_{1}} \cap E$, $S_{j}:=\gamma(\widetilde{G}_{i_{j}} \cap E \setminus \cup_{1 \le s < j}G_{i_{s}})$, $j > 2$ and $S:=\cup_{j}S_{j}$.
We put
$$
\widetilde{\underline{h}}(x):=\sum\limits_{j}(-1)^{j}\frac{(q-1)\epsilon}{2}\chi_{S_{j}}(x), \quad x \in \operatorname{X}.
$$
It follows from \eqref{eqq.66'} that $\widetilde{\underline{h}} \in \operatorname{LIP}_{1}(S)$. By Lemma \ref{Lm.McShane_Whitney} there is $\underline{h} \in \operatorname{LIP}_{1}(\operatorname{X})$ such that $\underline{h}|_{S}=\widetilde{\underline{h}}$.
Applying Proposition \ref{Prop.4.2} with $(\operatorname{Y},\rho)=(\mathbb{R},|\cdot|)$ and taking into account \eqref{eqq.65'} we deduce that $[\underline{h} \circ \gamma] \cap C([a,b]) = \emptyset$.
This leads to a contradiction with the assumptions of the lemma and concludes the proof.
\end{proof}

\begin{Remark}
One can show that
if the space $(\operatorname{X},\operatorname{d})$ satisfies the uniformly locally metrically doubling property (i.e., for each $R > 0$, there is $C(R) > 0$ such that,
for each ball $B_{R}(x)$, every $R/2$-separated subset of $B_{R}(x)$ consists of at most $C(R)$ different points), then the above assertion remains valid for $q=1$.
In the general case, we have essentially used that $q > 1$ in our proof.  Furthermore, we do not know, whether \eqref{eqq.62'} holds true with $q=1$ for a generic $\operatorname{X}$, but this seems unlikely.
\end{Remark}

\textit{Proof of Theorem \ref{Th.representative}.}
It will be convenient to split the proof into several steps.

\textit{Step 1.}
We fix an arbitrary sequence $\epsilon_{n} \subset (0,+\infty)$ such that $\epsilon_{n} \downarrow 0$, $n \to \infty$, and put
\begin{equation}
\label{eqq.67'}
\underline{E}:=\bigcap\limits_{n=1}^{\infty}\bigcap\limits_{i \in \mathcal{N}_{\epsilon_{n}}}E_{x_{\epsilon_{n},i}}.
\end{equation}
Since $[a,b]$ is compact, it is immediate from \eqref{eqq.67'} and Proposition \ref{Prop.4.3} that $\gamma(\underline{E})$ is precompact in $\operatorname{X}$.

\textit{Step 2.}
Assume that $[\gamma] \cap C([a,b]) = \emptyset$ and apply Proposition \ref{Prop.4.2} with $\operatorname{Y}=\operatorname{cl}(\gamma(\underline{E}))$, $f=\gamma$ and fix $\underline{\epsilon} > 0$
such that $\mathcal{L}^{2}(\mathcal{D}_{\gamma}(\underline{\epsilon},\delta)) > 0$ for every $\delta > 0$.
Let $\mathcal{N}$ be an arbitrary maximal $\underline{\epsilon}/8$-separated subset of $\operatorname{cl}(\gamma(\underline{E}))$. Since 
$\operatorname{cl}(\gamma(\underline{E}))$ is compact, we have
$N:=\operatorname{card}\mathcal{N} < +\infty$.

\textit{Step 3.}
It follows from the triangle inequality that, given $\delta > 0$,
\begin{equation}
\label{eqq.6.18}
\begin{split}
&\operatorname{d}(y,z) \geq \frac{\underline{\epsilon}}{4} \quad \hbox{for every $t',t'' \in \mathcal{D}_{\gamma}(\underline{\epsilon},\delta)$ and $y,z \in \mathcal{N}$ such that}\\
&\gamma(t') \in B_{\underline{\epsilon}/8}(y) \quad \hbox{and} \quad \gamma(t'') \in B_{\underline{\epsilon}/8}(z).
\end{split}
\end{equation}

\textit{Step 4.}
Given $y,z \in \mathcal{N}$ and $\delta > 0$, we put
$$
T_{y,z}(\underline{\epsilon},\delta):=\{(t',t'') \in \mathcal{D}_{\gamma}(\underline{\epsilon},\delta) \cap (\underline{E} \times \underline{E}):\gamma(t') \in B_{\underline{\epsilon}/8}(y), \gamma(t'') \in B_{\underline{\epsilon}/8}(z)\}.
$$
Since $N < +\infty$, there exist $\underline{y},\underline{z} \in \mathcal{N}$ such that
\begin{equation}
\label{eqq.bad_pairs}
\mathcal{L}^{2}(T_{\underline{y},\underline{z}}(\underline{\epsilon},\delta)) > 0 \quad \hbox{for every} \quad \delta > 0.
\end{equation}

\textit{Step 5.} Now we put
$$
\widetilde{h}(x):=\frac{\underline{\epsilon}}{8}\Bigl(\chi_{B_{\underline{\epsilon}/8}(\underline{y})}(x)+(-1)\chi_{B_{\underline{\epsilon}/8}(\underline{z})}(x)\Bigr), \quad x \in \operatorname{X}.
$$
It follows from \eqref{eqq.6.18} that $\widetilde{h} \in \operatorname{LIP}_{1}(V)$, where $V:=B_{\underline{\epsilon}/8}(\underline{y}) \cup B_{\underline{\epsilon}/8}(\underline{z})$. By Lemma \ref{Lm.McShane_Whitney},
there exists $h \in \operatorname{LIP}_{1}(\operatorname{X})$ such that $h|_{V}=\widetilde{h}$.
We have
\begin{equation}
\label{eqq.bad_pairs_2}
|h(\gamma(t'))-h(\gamma(t''))| = \frac{\underline{\epsilon}}{4} \quad \hbox{for each pair} \quad (t',t'') \in T_{\underline{y},\underline{z}}(\underline{\epsilon},\delta).
\end{equation}
Combining \eqref{eqq.bad_pairs}, \eqref{eqq.bad_pairs_2} and applying Proposition \ref{Prop.4.2} with $\operatorname{Y}=(\mathbb{R},|\cdot|)$, $f = h\circ\gamma$,
we obtain the equality $[h \circ \gamma] \cap C([a,b]) = \emptyset$. This contradiction completes
the proof.
\hfill$\Box$

Now we are ready to prove the second main result of this paper.

\textit{Proof of Theorem \ref{Th.main_2}.} By Lemma \ref{Lm.Borel_regularity}, $\gamma \in \mathfrak{B}([a,b],\operatorname{X})$.
By Theorem  \ref{Th.representative}, there is a set $E \subset [a,b]$ with $\mathcal{L}^{1}([a,b] \setminus E) =0$
such that the map $\gamma$ becomes continuous after a possible change on the set $[a,b] \setminus E$. In particular,
$\gamma$ is continuous on $E$. Hence, for every $h \in \operatorname{LIP}(\operatorname{X})$ the post-composition $h \circ \gamma$ is continuous on $E$.
Since $E$ is dense in $[a,b]$ and any continuous map is uniquely determined by its values on $E$, we immediately deduce that if $\overline{\gamma}$ is a unique element in
the set $[\gamma] \cap C([a,b],\operatorname{X})$, then
$h \circ \overline{\gamma}$ is the unique element in $[h \circ \gamma] \cap C([a,b])$. It remains to combine Proposition \ref{Prop.Sobolev_characterization} with Theorem \ref{Th.main_1}.
\hfill$\Box$

\end{document}